\documentclass[11pt]{article}

\usepackage{amsmath,amssymb,amsthm}
\usepackage[unicode,breaklinks=true,colorlinks=true]{hyperref}
\usepackage[dvipsnames]{xcolor}
\usepackage[normalem]{ulem} 

\usepackage[top=1in, bottom=1in, left=1.25in, right=1.25in, marginparwidth=1in, marginparsep=0.1in]{geometry}

\numberwithin{equation}{section}
\newtheorem{theorem}{Theorem}[section]

\newtheorem{lemma}[theorem]{Lemma}
\newtheorem{definition}[theorem]{Definition} 
\newtheorem{corollary}[theorem]{Corollary}

\theoremstyle{remark}
\newtheorem{remark}[theorem]{Remark}
\newtheorem{example}[theorem]{Example}

\definecolor{darkblue}{rgb}{0,0,0.7}

\usepackage{marginnote} 

\newcommand{\bke}[1]{\left( #1 \right)}

\newcommand{\bket}[1]{\left\{ #1 \right\}}
\newcommand{\norm}[1]{\| #1 \|}

\newcommand{\al}{\alpha}

\newcommand{\de}{\delta}
\newcommand{\e}{\epsilon}
\newcommand{\ga}{{\gamma}}

\newcommand{\la}{\lambda}

\newcommand{\om}{{\omega}}
\newcommand{\si}{\sigma}
\newcommand{\td}{\tilde}

\newcommand{\De}{\Delta}

\newcommand{\R}{{\mathbb R }}\newcommand{\RR}{{\mathbb R }}
\newcommand{\N}{{\mathbb N}}
\newcommand{\Z}{{\mathbb Z}}

\newcommand{\cN}{{\mathcal N}}

\newcommand{\pd}{{\partial}}
\newcommand{\nb}{{\nabla}}
\newcommand{\lec}{\lesssim}

\newcommand{\I}{\infty}
 
\renewcommand{\div}{\mathop{\mathrm{div}}}
\newcommand{\curl}{\mathop{\mathrm{curl}}}
\newcommand{\supp}{\mathop{\mathrm{supp}}}

\newcommand{\donothing}[1]{{}}

\newcommand{\EQ}[1]{\begin{equation}\begin{split} #1 \end{split}\end{equation}}
\newcommand{\EQN}[1]{\begin{equation*}\begin{split} #1 \end{split}\end{equation*}}

\DeclareMathOperator*{\esssup}{ess\,sup}

\makeatletter
\newcommand{\xRightarrow}[2][]{\ext@arrow 0359\Rightarrowfill@{#1}{#2}}
\makeatother

\newcommand{\loc}{\mathrm{loc}} 
\newcommand{\uloc}{\mathrm{uloc}}

\usepackage{accents}

\newcommand{\barT}{\overline T}
\newcommand{\pv} {\mathop{\mathrm{p.v.\!}}}

\begin{document}
\title{Global existence, regularity, and uniqueness of infinite energy solutions to the Navier-Stokes equations}
\author{Zachary Bradshaw and Tai-Peng Tsai}
\date{\today}
\maketitle 

\begin{abstract} 
This paper addresses several problems associated to local energy solutions (in the sense of Lemari\'e-Rieusset)
to the Navier-Stokes equations with initial data which is sufficiently small at large or small scales as measured using truncated Morrey-type quantities, namely:
(1) global existence for a class of data including the critical $L^2$-based Morrey space;
(2)  initial and eventual regularity of local energy solutions to the Navier-Stokes equations with initial data sufficiently small at small or large scales; (3) small-large uniqueness of local energy solutions for data in the critical $L^2$-based Morrey space. 
A number of interesting corollaries are included, including eventual regularity in familiar
Lebesgue, Lorentz, and Morrey spaces, a new local generalized Von Wahl uniqueness criteria, as well as regularity and uniqueness for local energy solutions with small discretely self-similar data.   

\end{abstract}


\section{Introduction}\label{sec.intro}

The Navier-Stokes equations describe the evolution of a viscous incompressible fluid's velocity field $u$ and associated scalar pressure $p$.  In particular, $u$ and $p$ are required to satisfy
\EQ{\label{eq.NSE}
&\partial_tu-\Delta u +u\cdot\nabla u+\nabla p = 0,
\\& \nabla \cdot u=0,
}
in the sense of distributions.  For our purpose, \eqref{eq.NSE} is applied on $\R^3\times (0,\I)$ and $u$ evolves from a prescribed, divergence free initial data $u_0:\R^3\to \R^3$.

In the classical paper \cite{leray}, J.~Leray constructed  global-in-time weak solutions to \eqref{eq.NSE} on $\R^4_+=\R^3\times (0,\infty)$ for any divergence free vector field $u_0\in L^2(\R^3)$.  Leray's solution $u$ satisfies the following properties:
\begin{enumerate}
\item $u\in L^\I(0,\I;L^2(\R^3))\cap L^2(0,\I;\dot H^1(\R^3))$,
\item $u$ satisfies the weak form of \eqref{eq.NSE},
\[
\iint -u \pd_t \zeta + \nb u:\nb \zeta + (u \cdot \nb )u \cdot \zeta = 0,\quad \forall \zeta \in C^\infty_c(\R^4_+;\R^3), \quad \div \zeta=0,
\]
\item $u(t)\to u_0$ in $L^2(\R^3)$ as $t\to 0^+$,
\item $u$ satisfies the \emph{global energy inequality}: For all $t>0$,
\[
\int_{\R^3} |u(x,t)|^2\,dx +2 \int_0^t \int_{\R^3} |\nabla u(x,t)|^2\,dx\,ds \leq \int_{\R^3} |u_0(x)|^2\,dx.
\]
\end{enumerate}
 The above existence result was extended to domains by Hopf in \cite{Hopf}.
We refer to the solutions constructed by Leray as \emph{Leray's original solutions} and refer to 
any solution satisfying the above properties as a \emph{Leray-Hopf weak solution}.   Note that, based on their construction, Leray's original solutions satisfy additional properties.  For example, they are suitable in the sense of \cite{CKN}; see \eqref{CKN-LEI}, this is proven in \cite[Proposition 30.1]{LR} and \cite{BCI}. Leray-Hopf weak solutions, on the other hand, are not known to be suitable generally. 

Although many important questions about these weak solutions remain open, e.g.,~uniqueness and global-in-time regularity, some positive results are available.  In particular, it is known that the singular sets of Leray-Hopf weak solutions which are suitable are compact in space-time.  This follows from Leray \cite[(6.4)]{leray}, and the partial regularity results of Scheffer \cite{VS76b} and Cafferelli, Kohn, and Nirenberg \cite{CKN} (see also \cite{LR} and \cite[Chap.~6]{Tsai-book}).

In his book \cite{LR}, Lemari\'e-Rieusset introduced a local analogue of suitable Leray-Hopf weak solutions called \emph{local energy solutions}.  These solutions evolve from uniformly locally square integrable data  $u_0\in L^2_{\uloc}$. Here, for $1\le q \le \infty$, $L^q_{\uloc}$ is the space of functions on $\R^3$ with finite norm
\[
\norm{u_0}_{L^q_{\uloc}} :=\sup_{x \in\R^3} \norm{u_0}_{L^q(B(x,1))}<\infty.
\]
We also denote
\[
E^q = \overline{C_c^\I(\R^3)}^{L^q_{\uloc}},
\]
the closure of $C_c^\I(\R^3)$ in $L^q_{\uloc}$-norm.
Having a notion of weak solution in a broader class than Leray's is useful when analyzing initial data in critical spaces such as the Lebesgue space $L^3$, the Lorentz space $L^{3,\I}=L^3_w$, or the Morrey space $M^{2,1}$, all of which embed in $L^2_{\uloc}$ but not in $L^2$ (see \cite{JiaSverak} for an example where this was crucial).  By \emph{critical spaces} we mean spaces for which the norm of $u$ is scaling invariant.  It is in such spaces that many arguments break down.  For example, $L^\I(0,T;L^3)$ is a regularity class for Leray-Hopf solutions \cite{ESS}, but this in unknown for $L^\I(0,T;L^{3,\I})$. 

The following definition is motivated by those found in \cite{LR,KiSe,JiaSverak-minimal,JiaSverak}.

\begin{definition}[Local energy solutions]\label{def:localLeray} Let $0<T\leq \I$. A vector field $u\in L^2_{\loc}(\R^3\times [0,T))$ is a local energy solution to \eqref{eq.NSE} with divergence free initial data $u_0\in L^2_{\uloc}(\R^3)$, denoted as $u \in \cN(u_0)$, if:
\begin{enumerate}
\item for some $p\in L^{3/2}_{\loc}(\R^3\times [0,T))$, the pair $(u,p)$ is a distributional solution to \eqref{eq.NSE},
\item for any $R>0$, $u$ satisfies
\begin{equation}\notag
\esssup_{0\leq t<R^2\wedge T}\,\sup_{x_0\in \R^3}\, \int_{B_R(x_0 )}\frac 1 2 |u(x,t)|^2\,dx + \sup_{x_0\in \R^3}\int_0^{R^2\wedge T}\int_{B_R(x_0)} |\nabla u(x,t)|^2\,dx \,dt<\infty,\end{equation}
\item for any $R>0$, $x_0\in \R^3$, and $0<T'< T $, there exists a function of time $c_{x_0,R}(t)\in L^{3/2}(0,T')$\footnote{The constant $c_{x_0,R}(t)$ can depend on $T'$ in principle. This does not matter in practice and we omit this dependence.} so that, for every $0<t<T'$  and $x \in B_{2R}(x_0)$  
\EQ{ \label{pressure.dec}
&p(x,t)=-\Delta^{-1}\div \div [(u\otimes u )\chi_{4R} (x-x_0)]
\\&\quad - \int_{\R^3} (K(x-y) - K(x_0 -y)) (u\otimes u)(y,t)(1-\chi_{4R}(y-x_0))\,dy 
+ c_{x_0,R}(t),
}
in $L^{3/2}(B_{2R}(x_0)\times (0,T'))$
where  $K(x)$ is the kernel of $\Delta^{-1}\div \div$,
 $K_{ij}(x) = \pd_i \pd_j \frac {-1}{4\pi|x|}$, and $\chi_{4R} (x)$ is the characteristic function for $B_{4R}$. 
\item for all compact subsets $K$ of $\R^3$ we have $u(t)\to u_0$ in $L^2(K)$ as $t\to 0^+$,
\item $u$ is suitable in the sense of Caffarelli-Kohn-Nirenberg, i.e., for all cylinders $Q$ compactly supported in  $ \R^3\times(0,T )$ and all non-negative $\phi\in C_c^\infty (Q)$, we have  the \emph{local energy inequality}
\EQ{\label{CKN-LEI}
&
2\iint |\nabla u|^2\phi\,dx\,dt 
\\&\leq 
\iint |u|^2(\partial_t \phi + \Delta\phi )\,dx\,dt +\iint (|u|^2+2p)(u\cdot \nabla\phi)\,dx\,dt,
}
\item the function
\[
t\mapsto \int_{\R^3} u(x,t)\cdot {w(x)}\,dx
\]
is continuous in $t\in [0,T)$, for any compactly supported $w\in L^2(\R^3)$.
\end{enumerate}
\end{definition}

For a given divergence free $u_0\in L^2_{\uloc}$, let $\mathcal{N}(u_0)$ denote the set of all local energy solutions with initial data $u_0$.

Our definition of local energy solutions is slightly different than the definitions from \cite{LR,KiSe,JiaSverak-minimal,JiaSverak}.  The definitions used in \cite{KiSe,JiaSverak-minimal,JiaSverak} require the data be in $E^2$,  which have some very mild decay at spatial infinity.  The pressure representation \eqref{pressure.dec} is replaced in \cite{JiaSverak-minimal,JiaSverak} by a very mild decay assumption on $u$, namely
\[
\lim_{|x_0|\to \I} \int_0^{R^2}\int_{B_R(x_0)} |u(x,t)|^2\,dx\,dt=0, \quad \forall R>0 .
\]
This condition implies a pressure representation like \eqref{pressure.dec} is valid (this is mentioned in \cite{JiaSverak-minimal} and explicitly proven in \cite{MaMiPr,KMT}).  If the data is only in $L^2_{\uloc}$, the above decay condition is unavailable and, therefore, we must build the pressure formula into the definition.   
In our arguments, the only reason to assume $u_0\in E^2$ would be to obtain the pressure formula \eqref{pressure.dec}.  To ensure full generality, it is thus better to assume \eqref{pressure.dec} explicitly and not impose decay on $u_0$.

Kikuchi and Seregin give another definition of local energy solutions in \cite{KiSe} which more closely resembles ours.  In \cite{KiSe}, \eqref{pressure.dec} is only assumed when $R=1$. Our definition is thus considerably stronger. Both definitions allow ``local energy estimates'' for $u_0\in L^2_{\uloc}$, but only ours leads to the estimate for all scales.

In  \cite{LR,LR-Morrey} (also see \cite{LR2}), Lemari\'e-Rieusset constructed 
local in time  local energy solutions if $u_0$ belongs to $L^2_\uloc$, and
global in time local energy solutions if $u_0$ belongs to $E^2$ or the Morrey space $M^{2,1}$ (see definition later in this section).
Kikuchi and Seregin \cite{KiSe} constructed global solutions for data in $E^2$ with more details and prove they satisfy the pressure formula in Definition \ref{def:localLeray} but with $R=1$.  
Recently, Maekawa, Miura, and Prange constructed local energy solutions on the half-space \cite{MaMiPr}.  This is a non-trivial extension of the whole-space case and required a novel treatment of the pressure. 
More recently, Kwon and Tsai \cite{KwTs} constructed global in time local energy solutions for non-decaying $u_0$ in $L^3_\uloc+ E^2$ with slowly decaying oscillation.   Also, Li constructed local energy solutions for the fractional Navier-Stokes equations \cite{Li}.

Naturally, less is known about local energy solutions than Leray's original solutions.  For example, Leray-Hopf weak solutions that satisfy the local energy inequality \emph{eventually regularize} in the sense that the set of singular times is compactly supported. Leray proved this in \cite[paragraph 34]{leray}, giving an upper bound of the set of singular times in \cite[(6.4)]{leray}.  Analogous results are currently unavailable for local energy solutions.  Indeed, it is speculated in \cite{BT1} that eventual regularity does not hold  for a discretely self-similar solution with $u_0 \in L^{3,\infty}(\R^3)$ if the solution has a local singularity.  Similarly, global existence is known in the Leray-Hopf class for any initial data in $L^2$, but is not known in the local energy class for any data in $L^2_\uloc$.

This paper is motivated by the problem of identifying similarities and differences between Leray-Hopf weak solutions and local energy solutions.  We address three subjects: eventual and initial regularity, global existence, and uniqueness. There are several themes that unify our resutls. First, 
our proofs are all based on the local energy methods in \cite{LR,JiaSverak-minimal,Jia-uniqueness}.  Second, the conditions in all of our results involve smallness of quantities closely associated with Morrey spaces. Our results shed light on the properties of local energy solutions with data in a variety of familiar function spaces as well as the regularity of discretely self-similar solutions to the Navier-Stokes equations.

For a solution $u$ in $\R^4_+$, we say that $(x,t)$ is a \emph{singular point} of $u$ if $ u \notin L^\I(B(x,r)\times(t-r^2,t))$ for any $r>0$. 
The set of all singular points is the \emph{singular set} of $u$.
We say that $t$ is a \emph{singular time} if there is a singular point $(x,t)$ for some $x$.  We say a solution $u$ has \emph{eventual regularity} if there is $t_1 < \infty$ such that $u$ is regular at $(x,t)$ whenever $t_1\le t$.  We say $u$ has \emph{initial regularity} if there exists $t_2$ such that $u$ is regular at $(x,t)$ whenever $0<t<t_2$. 

The following is our main theorem concerning eventual and initial regularity of solutions in the local energy class.

\begin{theorem}[Initial and eventual regularity]\label{thrm.littleo}
There exist small positive constants $\e_1$ and $c_0$ such that the following hold. Assume $u_0\in L^2_{\uloc}(\R^3)$, is divergence free and $u\in \mathcal N(u_0)$.  Let
\[
N_R^0 := \sup_{x_0\in \R^3} \frac 1 R \int_{B_R(x_0)}|  u_0|^2\,dx.
\]
\begin{enumerate}
\item 
If there exists $R_0>0$ so that
\begin{align}
\label{cond.infinity} \sup_{R\geq R_0} N^0_R< \e_1,
\end{align}
then $u$ has eventual regularity.  Moreover, if $3c_0R_0^2/4\leq t$, then 
\[
t^{1/2}\|u(\cdot,t)\|_{L^\I}\lesssim  ( \sup_{R\geq R_0} N^0_R)^{1/2} <\I.
\] 
\item If there exists $R_0>0$ so that
\begin{align}
\sup_{R\leq R_0} N^0_R< \e_1,
\end{align} 
then $u$ has initial regularity.   Moreover, if $t\leq c_0 R_0^2$, then
\[
t^{1/2}\|u(\cdot,t)\|_{L^\I}\lesssim (\sup_{R\leq R_0} N^0_R)^{1/2} <\I.
\]  
\item If $u_0$ satisfies 
\begin{align}\label{cond.everywhere} 
  \sup_{R>0} N^0_R< \e_1,
\end{align} 
then the set of singular times of $u$ in $\R^3\times (0,\I)$ is empty.  Moreover, for all $t> 0$,
\[
t^{1/2}\|u(\cdot,t)\|_{L^\I}\lesssim ( \sup_{R>0} N^0_R)^{1/2} <\I.
\]
\end{enumerate}
\end{theorem}

Note that $R_0$ depends on $u_0$ but is independent of $u\in \mathcal N(u_0)$.
Also note that Theorem \ref{thrm.littleo} does not assume $u_0 \in E^2$.

Conditions \eqref{cond.infinity}--\eqref{cond.everywhere} naturally lead us to consider initial data in Lorentz and Morrey spaces.  Recall that a vector field $f$ belongs to the Lorentz space $L^{p,q}$ for some $0<p<\I$ and $0<q\leq \I$ if, setting
\[
\|f\|_{L^{p,q}} := 
\begin{cases}
\bigg( \int_{0}^\I  \si^{q-1} | \{  x:|f(x)|>\si  \} |^{q/p} \,d\si   \bigg)^{1/q}& \text{ if }q<\I
\\   \sup_{\si>0} \big(  \si | \{  x:|f(x)|>\si  \} |^{1/p} \big) & \text{ if }q=\I
\end{cases} 
\] we have $\|f\|_{L^{p,q}}<\I$.  For $p>0$ and $s\leq n$, the
Morrey spaces $M^{p,s}$ contain vector fields such that 
\EQ{\label{Mps}
\|f\|_{M^{p,s}}:=\bigg(  \sup_{x_0\in \R^3} \sup_{R>0} \, \frac 1 {R^s}\int_{B_R(x_0)} | f|^p\,dx \bigg)^{1/p}<\I.
}
We also denote by $\td M^{p,s}$ the closure of $C^\infty_c$ in $M^{p,s}$-norm.
When we are only concerned with high frequencies, we can omit the low frequency behavior and consider the non-homogeneous Morrey spaces with norms
\EQ{\label{mps1}
\|f\|_{M^{p,s}_{\leq a}}:=\bigg(  \sup_{x_0\in \R^3} \sup_{0<R\leq a} \, \frac 1 {R^s}\int_{B_R(x_0)} | f|^p\,dx \bigg)^{1/p}<\I,
}
and similarly define $M^{p,s}_{< a}$ (of course, $M^{p,s}_{< a} = M^{p,s}_{\le a}$ and we will use these notions interchangeably). 
We refer to $M^{p,s}_{\leq 1}$ as $m^{p,s}$. 

Condition \eqref{cond.everywhere} means exactly that $\norm{u_0}_{M^{2,1}}^2 < \e_1$. A global regular solution for small data in $M^{2,1}$ is constructed by Kato in \cite{Kato} (see also Taylor \cite{Taylor}). Part 3 of Theorem \ref{thrm.littleo} asserts regularity for \emph{all} local energy solutions with $u_0$ sufficiently small in $M^{2,1}$ (or in $L^{3,\infty}$, as $L^{3,\infty}\subset M^{2,1}$, see Lemma \ref{lemma62}). Alternatively, this also follows from our uniqueness theorem below, Theorem \ref{thrm.uniquenessa}. 

For the Navier-Stokes equations, the most important examples of Lorentz or Morrey spaces are $L^{3,\I}$ and $M^{2,1}$.  These are critical spaces in the sense that they are dimensionless when computed for velocity fields.    Theorem \ref{thrm.littleo} leads to the following corollary on local energy solutions in familiar spaces.   
 
\begin{corollary}\label{cor.morrey}
Assume $u_0\in L^2_{\uloc}$ is divergence free and $u\in \cN(u_0)$.
\begin{enumerate}
\item If $u_0\in M^{2,r}$ where $0\leq r < 1$, then $u$ has eventual regularity and $t^{1/2}\|u(\cdot,t)\|_{L^\I}$ is bounded for sufficiently large $t$.  If $u_0\in M^{2,r}$ and $1<r\leq 3$, then $u$ has initial regularity  and $t^{1/2}\|u(\cdot,t)\|_{L^\I}$ is bounded for sufficiently small $t$.
\item  If $u_0\in \td M^{2,1} := \overline{C_c^\I}^{M^{2,1}}$, then $u$ has initial and eventual regularity and $t^{1/2}\|u(\cdot,t)\|_{L^\I}$ is bounded for sufficiently small and large $t$. In particular, this is true if $u_0\in L^{3,q}$ for $1\leq q<\I$.
\item If $u_0\in L^{p,q}$ where  $2<  p<3$ and $1\leq q \leq \I$ or if $u_0\in L^2$, then $u$ has eventual regularity and $t^{1/2}\|u(\cdot,t)\|_{L^\I}$ is bounded for sufficiently large $t$. If $u_0\in L^{p,q}$ where $3<p<\I$ and $1\leq q\leq \I$, then $u$ has initial regularity and $t^{1/2}\|u(\cdot,t)\|_{L^\I}$ is bounded for sufficiently small $t$.
\end{enumerate}    
\end{corollary}

The above corollary generalizes eventual regularity of Leray-Hopf weak solutions to a variety of new cases.
Note that for $2\leq q$,
\[
L^q \subset L^{q,s} \,( q\leq s \leq \I )\subset M^{2,r},
\]
where $r=3(1-2/q)$. 
Corollary \ref{cor.morrey}.1 thus applies to initial data  in $L^q$ where $2<q<3$, and the Lorentz scales $L^{q,s}$.   
The endpoint case $q=\I$ is beyond reach in part 2 because the test functions are not dense in $M^{2,1}$ (or even $L^{3,\I}$), a fact evidenced by $|x|^{-1}$. This is consistent with a remark in \cite{BT1} which proposes forward discretely self-similar solutions as counterexamples for eventual regularity.
Examples of solutions for data in $L^2$ are the Leray-Hopf weak solutions.  C.~Calderon constructed weak solutions for data in $L^q$ when $2<q<3$ in \cite{Calderon}.

Using Theorem \ref{thrm.littleo}, we also obtain a new small data regularity criteria for discretely self-similar solutions in the local energy class. 
Recall that solutions to \eqref{eq.NSE} satisfy a 
natural scaling: if $u$ satisfies \eqref{eq.NSE}, then for any $\lambda>0$
\begin{equation}
u^{\lambda}(x,t)=\lambda u(\lambda x,\lambda^2t),
\end{equation}
is also a solution with pressure 
\begin{equation}
p^{\lambda}(x,t)=\lambda^2 p(\lambda x,\lambda^2t),
\end{equation}
and initial data 
\begin{equation}
u_0^{\lambda}(x)=\lambda u_0(\lambda x).
\end{equation}
A solution is called self-similar (SS) if $u^\lambda(x,t)=u(x,t)$ for all $\lambda>0$ and is discretely self-similar with factor $\lambda$ (i.e.~$u$ is $\lambda$-DSS) if this scaling invariance holds for a given $\lambda>1$. Similarly, $u_0$ is self-similar (a.k.a.~$(-1)$-homogeneous) if $u_0(x)=\lambda u_0(\lambda x)$ for all $\lambda>0$ or $\lambda$-DSS if this holds for a given $\lambda>1$.  
These solutions can be either forward or backward if they are defined on $\R^3\times (0,\I)$ or $\R^3\times (-\I,0)$ respectively.  We focus on the forward case.  Forward self-similar and DSS solutions are known to exist for SS or DSS data in a variety of function spaces \cite{Barraza,BT1,BT2,BT3,BT5,CP,Chae-Wolf,GiMi,JiaSverak,Kato,KT-SSHS,LR2,Tsai-DSSI}, but, for large data, their fine properties have not been thoroughly investigated (for small data, see \cite{Brandolese}).   
Gruji\'c proved the only existing result in this direction in \cite{Grujic}, showing that any forward self-similar solution in the local energy class is smooth.  This is, in general, not known for forward \emph{discretely} self-similar solutions. Indeed, Gruji\'c's argument breaks down for DSS solutions because their singular sets might possess isolated singularities in space-time, which is not ruled out in \cite{CKN}.  A self-similar solution, on the other hand, would have at least a 1 dimensional (in space-time) singular set which violates conditions in \cite{CKN}.  Smoothness has recently been established in \cite{KMT} when $u_0\in L^{3,\I}$ is $\la$-DSS and $\la$ is close to $1$.
Our next result establishes smoothness for discretely self-similar solutions evolving from small initial data in $L^2_\uloc$. Note that, solutions are known to exist for such data \cite{BT1,BT2,BT5}.  

\begin{corollary}[Regularity of small-data DSS solutions]\label{cor.DSS1} 
Assume $u_0\in L^2_{\uloc}$  is divergence free and $\la$-DSS for some $\la> 1$, and that $u\in \cN(u_0)$.  If  $\|u_0\|_{L^2_{\uloc}}<\e_1/\sqrt \la$, then $u \in C^\I_{\loc}(\R^3\times \R_+)$.
\end{corollary}

Note that we do not require $u$ to be DSS in the statement of the theorem. When $u_0$ is DSS, its $M^{2,1}$ and $L^2_\uloc$ norms are equivalent, see \eqref{DSSL2uloc}. So smallness in one implies smallness in the other.  We will later establish uniqueness in $\mathcal N(u_0)$ for the same data in Corollary \ref{cor.DSS2}. 
However, DSS $u_0$ in $L^2_{\uloc}$ may not be in $E^2$; see Lemma \ref{lemma62}. The only currently available existence results for such $u_0$ in \cite{Chae-Wolf,BT5} give us DSS solutions but not the  local pressure decomposition.   However, when the initial data belong to $L^2_\uloc$, it is not difficult to prove the solutions constructed in \cite{BT5} are local energy solutions, implying the unique $u \in \cN(u_0)$ is DSS. We will revisit this in Section \ref{sec.existence.DSS}.

\bigskip 
We next turn our attention to the problem of global existence for some possibly non-decaying data in $L^2_\uloc$.   

\begin{theorem}[Global existence]\label{thrm.existence}
Assume $u_0\in L^2_{\uloc}$, is divergence free, and 
\EQ{\label{condition.forexistence}
\lim_{R\to \I} \sup_{x_0\in \R^3} \frac 1 {R^2} \int_{B_R(x_0)} |u_0(x)|^2\,dx = 0.
}
Then, there exists a global in time local energy solution $u$ to \eqref{eq.NSE} with initial data $u_0$.
\end{theorem}

In particular, any divergence-free $u_0\in M^{2,1}(\R^3)$ satisfies the conditions in Theorem \ref{thrm.existence}, while $u_0$ may not be in $E^2$; see Lemma \ref{lemma62}.

For large initial data, the existence of global in time solutions in critical spaces related to $M^{2,1}$, namely $L^3,L^{3,\I},$ and $\dot B^{-1+3/p}_{p,\I}$ where $p<\I$ has recently been studied in \cite{SeSv,BaSeSv,AlBa}.  It is unknown if global in time weak solutions exist for data in the Koch-Tataru space $BMO^{-1}$ (see \cite{KochTataru,LR}).
Note that, unless the solution has some special structure (see, e.g., \cite{Tsai-DSSI,BT1,Chae-Wolf,BT5,LR2}), most global-in-time results assume something is decaying at spatial infinity.  This could be, for instance, that $u_0\in E^2$ \cite{KiSe}, that $u_0$ has decaying oscillation \cite{KwTs}, that $u_0$ is in a stronger space than $E^2$ like $L^3$ or $L^{3,\I}$ \cite{SeSv,BaSeSv} or that $u_0$ is in a non-endpoint Besov space with which is scaling invariant for the Navier-Stokes problem \cite{AlBa} (these spaces still have decay since each Littlewood-Paley block is in $L^p$ and $p<\I$).
Decay at spatial infinity allows a local in time solution to be split at a positive time into a
part which is small in a dimensionless space and a large finite energy part.  The solution is then extended in time by gluing together a local strong solution (the time scale of which is uniform due to smallness), and a weak solution to a perturbed problem.  The only example where the splitting argument is not used is the case of $M^{2,1}$ in \cite{LR-Morrey}, which is a special case of our more general result, Theorem \ref{thrm.existence}.

To prove Theorem \ref{thrm.existence}, we use ideas from \cite{JiaSverak-minimal} to extend a priori bounds starting at the initial data to arbitrarily large times directly by passing to larger and larger scales.   This is different than the usual approach since smallness at spatial infinity does not play a role.  Note that in some regard, we are still assuming some weak form of decay at spatial infinity since a constant function does not satisfy \eqref{condition.forexistence}. Let us mention that Lemari\'e-Riuesset's proof for the special case of data in $M^{2,1}$  \cite{LR-Morrey} is similar to ours, but we were not aware of it until after writing this paper.

\bigskip 

The last results in this paper concern the uniqueness of solutions in $\mathcal N(u_0)$. We include a global and local result when $u_0$ is small in some sense.

\begin{theorem}[Uniqueness for small data in $M^{2,1}$]\label{thrm.uniquenessa}
Assume $u_0\in L^{2}_\uloc$ and is incompressible.   Let $u$ and $v$ be elements of $\cN(u_0)$.  There exists a universal constant $\e_2$ such that,  if $\|u_0\|_{M^{2,1}}\leq \e_2$, then $u=v$ as distributions on $\R^3\times (0,\I)$.
\end{theorem}

\begin{theorem}[Uniqueness for data that is small at high frequencies]\label{thrm.uniquenessb}
Assume $u_0\in L^{2}_\uloc$ and is incompressible.   Let $u$ and $v$ be elements of $\cN(u_0)$.  There exists a universal constant $\e_2$ such that, 
  if  either $u_0$ satisfies \eqref{condition.forexistence} or $u_0\in E^2$,
and $\lim_{R\to 0} N^0_R<\e_2$, then there exists $T>0$ so that $u=v$ as distributions on $\R^3\times (0,T)$.
Furthermore, $T\sim R^2$  where $R>0$ satisfies
\[
\sup_{0<r\leq R} N^0_r \leq \e_2.
\]
\end{theorem}

In particular, if $u_0\in M^{2,1}$, then $u_0$ satisfies \eqref{condition.forexistence} and Theorem \ref{thrm.uniquenessb} is applicable.

Theorem \ref{thrm.uniquenessa} is motivated by Jia \cite{Jia-uniqueness} who established uniqueness for local energy solutions with small data in $L^{3,\I}$.  Our proof mainly  follows his argument, although going from $L^{3,\I}$ to $M^{2,1}$ introduces some technical hurdles. Lemari\'e-Rieusset includes a similar theorem in \cite[Theorem 2]{LR-Morrey}. We note that our result is an improvement because Lemari\'e-Rieusset's assumptions imply $\lim_{R\to 0} N^0_R = 0$ while we allow this to be positive but small. Hence our result may include small SS or DSS data.  Furthermore, the only solutions considered in \cite[Theorem 2]{LR-Morrey} are the limits of the regularized system, while ours come from a more general class.

It is interesting to note that Morrey spaces and local energy methods have recently played a role in \cite{LR-unique} in answering an interesting question of T.~Barker \cite{barker} concerning local uniqueness of suitable weak solutions with data in $L^2 \cap X$ where $X$ is a subspace of $BMO^{-1}$ which imposes some smoothness on the data.

Let us remark that combining Theorems \ref{thrm.littleo}, \ref{thrm.existence}, and \ref{thrm.uniquenessa} yields a global well-posedness result reminiscent of \cite{Kato}  for small data in $M^{2,1}$ but is proved using an entirely different method (see also \cite{LR-Morrey,Taylor}). Their solutions live in $L^\infty(0,\infty;M^{2,1})$ while ours are local energy solutions.

As a corollary of Theorem \ref{thrm.uniquenessb}, we obtain  local in time uniqueness of local energy solutions with initial data in $E^3$, which is the closure of $C_c^\I$ in the $L^3_{\uloc}$ norm.
This gives an alternative proof of the uniqueness part of \cite[Theorem 33.2]{LR}.

\begin{corollary}[Local uniqueness in $E^3$]\label{cor.uniqueness} 
Assume $u_0\in E^3$ and is divergence free.  Let $u$ and $v$ be elements of $\cN(u_0)$. Then, there exists $T=T(u_0)>0$ so that  $u=v$ as distributions on $\R^3\times (0,T)$. 
\end{corollary}
In the preceding corollary, $T$ only depends on $u_0$, and the smallness assumption is hidden in the spatial decay of $u_0$. This result is not new, but our proof is and we include it to emphasize the usefulness of the arguments. Note that it also follows from \cite[Theorem 2]{LR-Morrey}.
 
We can go further concerning uniqueness problems. In \cite{LR-Morrey}, Lemari\'e-Rieusset stated a  problem concerning uniqueness of certain solutions in $C([0,T];\td M^{2,1})$ where $\td M^{2,1}$ denotes the closure of $C_c^\I$ in $M^{2,1}$: 

\medskip {\bf Problem:} If $u_0\in L^2$ is divergence free and $u\in L^\I((0,T);L^2)\cap L^2((0,T);H^1) \cap C ([0,T] ; \td M^{2,1} )$ along with a pressure $p$ solve the Navier-Stokes equations, then is $u$ a Leray solution (in the sense of \cite[Definition 2]{LR-Morrey}) and moreover, is it the unique Leray solution?  
 
 \medskip 
This problem appears to have been answered affirmatively in  \cite{LR-Besov} where a more general uniqueness result is given in $ C ([0,T] ; \td B^{-1}_{\I,\I})$.  Theorem \ref{thrm.uniquenessb} gives another way of addressing this question. In fact, we address a more general localized version of the problem that appears to be new.   

Let us introduce some notation. Recall $m^{2,1}$ is defined by \eqref{mps1}.
Let $m^{2,1}_\e$ be the collection of $f\in m^{2,1}$ so that  
\[
\limsup_{r\to 0^+ }\sup_{x_0\in \R^3}   \frac 1 {r} \int_{B_r(x_0)} |f|^2\,dx \leq \e.
\]
We also let $\td m^{2,1}$ denote the closure of the test functions in $m^{2,1}$. 
It is a strict subset of $m^{2,1}_0$ as $\td m^{2,1}\subset E^2$;  see Remark \ref{lemma.morreyspacedecay-rmk}.
 
Using Theorem \ref{thrm.uniquenessb} we are able to prove the following theorem.

\begin{theorem}[Weak-strong uniqueness]\label{thrm.uniqueness2}
Fix $T\in (0,\I)$. Let $u_0\in E^2$ be divergence free. Let $\e_2$ be as in Theorem \ref{thrm.uniquenessb}.   Let $0<\e<\e_2$ be given. 
Assume $u \in \mathcal N(u_0)\cap C([0,T) ;m_{\e}^{2,1})$ and $v\in \mathcal N(u_0)$.  Then $u=v$ as distributions on $\R^3\times (0,T)$.
\end{theorem}

There is also a weak-strong uniqueness result in \cite[Theorem 14.7]{LR2} for solutions that can be split into a small part in a critical multiplier space which embeds strictly into $M^{2,1}$ and a non-critical part (see \cite[p. 94]{LR2}).

Clearly $\td m^{2,1}\subset m^{2,1}_\e$ for any $\e>0$ (see Lemma \ref{lemma.morreyspacedecay}). However, $u\in C([0,\I) ; \td m^{2,1})$ implies several of the items from the definition of local energy solution. It is thus not difficult to arrive at a corollary to Theorem \ref{thrm.uniqueness2} where sufficient conditions for $u\in \mathcal N(u_0)$ are hidden in the assumption that $u\in C([0,\I) ; \td m^{2,1})$.
In fact, this is a local analogue to the problem given by Lemari\'e-Rieusset in \cite{LR-Morrey}.
 
\begin{corollary}[Generalized Von-Wahl uniqueness criteria]\label{cor.uniqueness2}
Let $u_0\in L^2_\uloc$ be divergence free.
Assume $u\in C([0,\I) ; \td m^{2,1})$, there {exists a pressure $p\in L^{3/2}_\loc(\R^3\times [0,\infty))$} so that $(u,p)$ solve \eqref{eq.NSE} as distributions and for all $T<\I$,
\[
\sup_{x_0\in \R^3} \int_0^T \int_{B_2(x_0)} |\nb u|^2\,dx \,dt <\I.
\]
Then $u\in \mathcal N(u_0)$ and is unique in $\mathcal N(u_0)$.

\end{corollary}

We can also use Theorem \ref{thrm.uniquenessa} to show $\la$-DSS solutions are unique provided the initial data is small in $L^2_\uloc$.

\begin{corollary}[Uniqueness of small-data DSS solutions]\label{cor.DSS2}
Assume $u_0\in L^2_{\uloc}$  is divergence free and $\la$-DSS for some $\la> 1$, and that $u\in \cN(u_0)$.  If  $\|u_0\|_{L^2_{\uloc}}<\e_2/\sqrt\la$, then $u$ is unique in $\cN(u_0)$.  
\end{corollary}

As a concluding remark, note that, although we are considering several different problems concerning local energy solutions, there are two unifying themes that recur throughout this paper. The first is that all main results connect Morrey-type norms or truncations of these norms to small or large scales to the analysis of the Navier-Stokes equations, highlighting the importance of these quantities. The second is that all the main results rely on a tremendously useful a priori bound which was discovered by Lemari\'e-Rieusset and later explicitly extended to all scales in \cite{JiaSverak-minimal} (see inequality \eqref{ineq.apriorilocal}; the first use of local energy methods at large scales appears to be in \cite{LR-Besov}).

The paper is arranged as follows.  Theorems \ref{thrm.littleo} and \ref{thrm.existence} are proven respectively in Sections \ref{sec.proofs1} and \ref{sec.existence}, while Theorems \ref{thrm.uniquenessa} and \ref{thrm.uniquenessb} are proven in Section \ref{sec.unique}.  The corollaries are proved in Section \ref{sec.proofs2}. An appendix is included as Section \ref{sec.appendix} to illustrate the relationships between several function spaces appearing in this paper.

In the end of the introduction, we consider a related concept of ``\emph{far-field regularity}'' which means that, for any finite $T>0$, there is a large $R$ such that the solution is regular in $\bket{(x,t)\in \R^3 \times (0,T):\ |x|>R}$. This property is well-known for weak solutions of \eqref{eq.NSE} in the energy class. For local energy solutions with $u_0 \in E^2$, this can be derived using an argument of \cite[page 354]{LR} based on the $\e$-regularity criterion of Caffarelli-Kohn-Nirenberg (see Theorem \ref{thrm.epsilonreg}): Specifically, using Theorem \ref{thrm.epsilonreg} and the decay
\[
\lim_{|x_0|\to \infty} \int_0^T \int_{B_1(x_0)}|u|^3 + |p-c_{x_0,1}(t)|^{3/2}\,dx\,dt=0
\]
(see  \cite[Proposition 32.2]{LR} and \cite[Lemma 2.2]{KiSe}), we can show for any $0<t_1<t_2<T$ that $u \in L^\infty((t_1,t_2)\times B_{R_0}^c)$ for $R_0$ sufficiently large.  See \cite[Corollary 4.8]{KwTs} for details of its application that $u(t)\in E^3$ for a.e.~$t$,  and \cite{AlBa} for an extension for Besov space data.

It is unclear if far field regularity holds in classes where there is no decay (in the $E^2$ sense) at spatial infinity, e.g.~$M^{2,1}$.   Consider for example an initial data that looks like 
\EQ{\label{eq1.13}
f(x)=\sum_{k\in \Z}  \frac {\chi_{B_{1/8}(0)} (x-ke_1)  } {|x-ke_1|}.
}
Then $f(x)\in M^{2,1} \setminus E^2$ and, based on the periodicity in the $e_1$ direction, far-field regularity is equivalent to regularity. This suggests that far-field regularity may fail for $u_0\in M^{2,1}$.


When this manuscript is near completion, Fern\'andez-Dalga and Lemari\'e-Rieusset released an interesting paper \cite{FDLR} addressing global existence in a general context related to Theorem \ref{thrm.existence}. Our Theorem \ref{thrm.existence} is independent of their work, and has been presented in the Nonlinear Analysis seminar in Rutgers University on April 9, 2019, in a plenary lecture of the 
International Congress of Chinese Mathematicians on June 13, 2019, in Tsinghua University, Beijing, and in Henan University, Kaifeng, on June 16, 2019.

\section{Eventual and initial regularity}\label{sec.bound}\label{sec.proofs1}

In this section we prove Theorem \ref{thrm.littleo}. There are two main ingredients, an a priori estimate in \cite{JiaSverak-minimal} and a version of the Cafarelli-Kohn-Nirenberg regularity criteria. We recall both as lemmas.

\begin{lemma}\label{lemma.JiaSve}
Let $u_0\in L^2_\uloc$, $\div u_0=0$, and assume $u\in \mathcal N (u_0)$.  For all $r>0$ we have
\begin{equation}\label{ineq.apriorilocal}
\esssup_{0\leq t \leq \sigma r^2}\sup_{x_0\in \RR^3} \int_{B_r(x_0)}\frac {|u|^2} 2 \,dx\,dt + \sup_{x_0\in \RR^3}\int_0^{\sigma r^2}\int_{B_r(x_0)} |\nabla u|^2\,dx\,dt <CA_0(r) ,
\end{equation}

\begin{equation}\label{ineq.apriorilocal2}
\sup_{x_0\in \RR^3} \int_0^{\sigma r^2}\!\!\int_{B_r(x_0) }\big( | u|^3  +|p-c_{x_0,r}(t)|^{3/2}  \big)\,dx\,dt
 <C r^{\frac 12} A_{0}(r)^{\frac 32},
\end{equation}
where
\[
A_0(r)=rN^0_r= \sup_{x_0\in \R^3} \int_{B_r(x_0)} |u_0|^2 \,dx,
\] 
and
\begin{equation}\label{def.sigma}
\si=\sigma(r) =c_0\, \min\big\{(N^0_r)^{-2} , 1  \big\},
\end{equation}
for a small universal constant $c_0>0$.  
\end{lemma}

See \cite[Lemma 3.5]{KMT} for revised \eqref{ineq.apriorilocal2} with higher exponents.

As mentioned in Section 1, the solutions in \cite{JiaSverak-minimal} are defined differently than they are here--we only require $u_0 \in L^2_{\uloc}$ and do not require $u_0 \in E^2$, and therefore assume \eqref{pressure.dec} explicitly.  Inspecting  \cite[Proof of Lemma 2.2]{JiaSverak-minimal}, however, reveals that the same conclusion is valid for our local energy solutions.  
The only difference is that our solutions are not decaying.
In \cite{JiaSverak-minimal}, decay is used to ensure the local pressure expansion is satisfied and that $A(\la)$ is continuous in $\la$ (see \cite[Page 1452 top]{JiaSverak-minimal}).  For us, the local pressure expansion is built into Definition \ref{def:localLeray}, but continuity is unclear when $u_0\in L^2_\uloc\setminus E^2$.   
To prove Lemma \ref{lemma.JiaSve} without continuity, we need the following version of Gr\"onwall's lemma.

\begin{lemma}\label{lemma.gronwall}Suppose $f(t) \in L^\infty_\loc([0,T); [0,\infty))$ satisfies, for some $m \ge 1$,
\[
f(t) \le a + b\int_0^t( f(s) + f(s)^m) ds, \quad 0<t<T,
\]
where $a,b>0$,
then for $T_0=\min(T,T_1)$, with $T_1$ defined by \eqref{T1.def}, we have $f(t) \le 2a$ for $t \in (0,T_0)$.
\end{lemma}

Note $f$ may be discontinuous. 
\begin{proof} By replacing $f(t)$ by $\td f(t)=\esssup_{s <t } f(s)$, we may assume $f$ is nondecreasing. Let $g(t)$ be the solution of
\[
g(t) = \frac 54 a + b\int_0^t (g(s) + g(s)^m )ds, \quad 0<t<T_1.
\]
$T_1$ is such that
\EQ{\label{T1.def}
 b\int_0^{T_1}( (2a) + (2a)^m )ds = \frac 34a.
}
We have $g \in C^1$, $g(t) \le 2a$ in $[0,T_1]$, and $f(t) < g(t)$ for sufficiently small $t$. Let 
\[
t_2 = \sup\bket{t \in (0,T_0): f(s) \le g(s), \quad \forall s \in (0,t)}.
\]
We have $t_2>0$.
If $t_2=T_0$, we are done. If $t_2 \in (0,T_0)$, let $t_3=\frac12(t_2,T_0)$, $M = \norm{f}_{L^\infty(0,t_3)}$, and we can choose $t_4 \in (t_2,t_3)$ so that $(t_4-t_2 ) b(M+M^m)\le \frac a8$ and $f(t_4) > g(t_4)$ by the definition of $t_2$.
Then
\EQN{
f(t_4) & \le a + b\int_0^{t_4} (f(s) + f(s)^m )ds
\\ & \le   a + b\int_0^{t_2} (g(s) + g(s)^m )ds 
+ b\int_{t_2}^{t_4} (M+M^m)ds
\\ & \le g(t_2) - \frac a4 + \frac a8,
}
which is a contradiction.
\end{proof}

\begin{proof}[Proof of Lemma \ref{lemma.JiaSve}]
We use essentially the same estimates as  in the \cite[Proof of Lemma 2.2]{JiaSverak-minimal}.  By H\"older and Young inequalities, for any $\de>0$,
\[
\norm{u}_{L^3L^3}^3\lec \norm{u}_{L^6L^2}^{3/2}\norm{u}_{L^2L^6}^{3/2}\lec (\de R)^{-3}\norm{u}_{L^6L^2}^{6}+ \de R\norm{u}_{L^2L^6}^{2}.
\]
Thus, also by Sobolev inequality,
\EQ{
& \frac 1 {R} \int_0^{\si R^2} \int_{B_{2R}(x_0)} |u|^3 \,dx\,dt
\\&\leq  \frac C {\de^3 R^4} \int_0^{\si R^2}\bigg( \int_{B_{2R}(x_0)} |u|^2 \,dx\bigg)^{3}\,dt + \frac {C \de}{R^{2}} \int_0^{\si R^2}  \int_{B_{2R}(x_0)} |u|^2 \,dx\,dt 
\\&\quad + C\de \sup_{x_0\in \R^3} \int_0^{\si R^2} \int_{B_{2R(x_0)}} |\nb u|^2\,dx\,dt,
}
with $C$ independent of $\si$.
For the pressure, using \eqref{pressure.dec} we have   
\EQ{
& \frac 1 R \int_0^{\si R^2} \int_{B_{2R}(x_0)} |p - c _{x_0,R}(t)|^{3/2}\,dx\,dt 
\\&\leq \frac C R  \int_0^{\si R^2} \int_{B_{4R}(x_0)} |u|^3 \,dx\,dt +\int_0^{\si R^2}   \frac C {R^{4}} \bar A(\si)^{3/2}  \,dt,
}
where 
\EQ{
\bar A(\si) = \esssup_{0\leq t\leq \si R^2} \sup_{x_0\in \R^3} \int_{\R^3} \frac {|u|^2} 2 \phi(x-x_0)\,dx.
}
Now, adopting the same terminology as in \cite[Proof of Lemma 2.2]{JiaSverak-minimal} and working from the local energy inequality we obtain
\EQ{
&\int_{\R^3} \frac {|u|^2} 2 \phi(x-x_0) \,dx + \int_0^t \int_{\R^3} |\nb u|^2 \phi(x-x_0)\,dx\,ds 
\\&\leq \al + C \frac 1 {R^2}  \int_0^{\si R^2} \bar A(\si) \,ds+C  \frac 1 {R^4}    \int_0^{\si R^2} \bar A(\si)^3\,ds,
}
where we  chose sufficiently small $\de$,  
defined $\al$ as in \cite{JiaSverak-minimal}, 
and handled the linear term in the obvious way.  Hence
\EQ{
& \frac {\bar A (\si)} R \leq \frac \al R + \frac C {R^2}  \int_0^{\si R^2} \frac { \bar A(\si)} {R} \,ds+ \frac C {R^2 }   \int_0^{\si R^2}\bigg(\frac {\bar A(\si)} R   \bigg)^3.
}
We now use Lemma \ref{lemma.gronwall} to obtain
\[
\bar A (\si) \leq 2 \al,
\]
for $t\in [0,T_R]$ where $T_R = \si R^2$ and 
\[
\si = c_0\min\{(N_R^0)^{-2},1\}
\]
for an appropriately chosen small constant $c_0$ that is independent of $R$ and $u_0$.  This constant is chosen so that 
\[
c_0\min\{(N_R^0)^{-2},1\} \sim \frac C {2+ \al^2/R^2}.
\]
The remaining conclusions follow as in \cite{JiaSverak-minimal}.  
\end{proof}

We will use the following $\e$-regularity criteria which is motivated by \cite{CKN}. The current revised form is due to \cite{L98}; see also \cite{LS99} for details.
  
\begin{lemma}[$\e$-regularity criteria]\label{thrm.epsilonreg}  
There exists a universal small constant $\e_*>0$ such that, if the pair $(u,p)$ is a suitable weak solutions of \eqref{eq.NSE} in  $Q_r=Q_r(x_0,t_0)=B_r(x_0)\times (t_0-r^2,t_0)$, $B_r(x_0)\subset \R^3$, and    
\[
{\e^3=}\frac 1 {r^2} \int_{Q_r} (|u|^3 +|p|^{3/2})\,dx\,dt <\e_*,
\]
then $u\in L^\I(Q_{r/2})$.
Moreover,
\[
\|  \nabla^k u\|_{L^\I(Q_{r/2})} \leq C_k {\e}\, r^{-k-1},
\]
for universal constants $C_k$ where $k\in \N_0$.
\end{lemma}

 \begin{proof}[Proof of Theorem \ref{thrm.littleo}]  Assume there exists $R_0>0$ so that for all $R\geq R_0$, $N^0_R<\e_1$.
 We will give $\e_1\in(0,1)$ a precise value later in the proof.

Fix $x_0\in \R^3$ and $R>R_0$.  Let $\tilde p(x,t)= p(x,t)-c_{x_0,R}(t)$ where $c_{x_0,R}(t)$ is the function of $t$ from formula \eqref{pressure.dec}.  Then $u$ is a suitable weak solution to the Navier-Stokes equations with associated pressure $\tilde p$.  
By \eqref{ineq.apriorilocal2}, we have
\begin{align*}
\int_0^{\si(R) R^2} \int_{B_{ R}(x_0)} (|u|^3 +|\tilde p|^{3/2})\,dx\,dt &\leq C {(N^0_R)}^{3/2}  R^{2}.
\end{align*}
By  \eqref{def.sigma} and  $N^0_R<\e_1<1$, $\si(R)=c_0<1$.
Dividing by $c_0R^2$,
\[
\frac 1 {c_0R^2} \int_0^{c_0 R^2} \int_{B_{c_0^{1/2} R}(x_0)} (|u|^3 +|\tilde p|^{3/2})\,dx\,dt  \leq  \frac {C (N^0_R)^{3/2}} {c_0} \leq \frac {C \e_1^{3/2}} {c_0}.
\]
Thus, provided $R\geq R_0$  and $\e_1 \leq (c_0C^{-1} \e_* )^{2/3}$, the right side is bounded by $\e_*$ and we have by Lemma \ref{thrm.epsilonreg} that 
\[
u\in L^\I(Q), \quad Q=B_{c_0^{1/2}R/2}(x_0)\times [3c_0R^2/4, c_0R^2],
\]
and for $(x,t) \in Q$,
\EQ{\label{eq2.4}
|u(x,t)| \le C_0 (\frac C{c_0}(N_R^0)^{3/2})^{1/3} (c_0^{1/2}R/2)^{-1} \le C (N_R^0)^{1/2} t^{-1/2}.
}
Thus $u$ is regular in $\R^3 \times (3c_0R^2/4, c_0R^2]$.
Since $R\ge R_0$ is arbitrary, $u$ is regular at $(x,t)$ for any $x \in \R^3$ and $t> 3c_0R_0^2/4$, with the bound \eqref{eq2.4}. Note that $3c_0R_0^2/4$ is determined by $u_0$ and is the same for all $u \in \cN (u_0)$.

The proof is similar when $\sup_{R\leq R_0}N^0_R<\e_1$ and we omit the details.

Finally, assume $N^0_R<\e_1$ for all $R>0$. 
Then $\si(R)= c_0$ for all $R>0$, and $u$ is regular with the bound \eqref{eq2.4} in 
\[
\bigcup_{0<R<\I} \R^3  \times (3c_0R^2/4, c_0R^2] = \R^3  \times (0,\I).\qedhere
\]
\end{proof}

\section{Global existence}\label{sec.existence} 

In this section we prove Theorem \ref{thrm.existence}.  We will first construct solutions to a regularized system in subsection \ref{sec.existence.1}, and then take limits in subsection \ref{sec.existence.2}. As in \cite{LR-Morrey}, the solution will be constructed for $0<t<\infty$ in one step, and there is no need of an extension argument as in \cite{LR,KiSe,KwTs}.

\subsection{Global existence for a regularized system}\label{sec.existence.1}

 The goal of this subsection is to construct global in time solutions to the regularized system 
\EQ{\label{eq.NSEreg}
&\partial_t u^\e - \Delta u^\e +  (  \mathcal J_\e(u^\e) \cdot \nb  )(u^\e \Phi_\e) +\nb p^\e =0
\\& \div u^\e =0,   
}
when $u_0$ satisfies \eqref{condition.forexistence}, $\mathcal J_\e f= \eta_\e *f$ for a mollifier $\eta_\e$ and $\Phi_\e(x)=\Phi(\e x)$ for a fixed radially decreasing cutoff function $\Phi$ that equals $1$ on $B_1(0)$ and $\supp \Phi\subset B_{3/2}(0)$.  
This system was studied in \cite[Section 3]{KwTs} and we recall and combine \cite[Lemmas 3.3 and 3.4]{KwTs} in the following lemma.   
\begin{lemma} \label{lemma.KwTs}
Let $u_0\in L^2_\uloc$ with $\div u_0=0$ and $\|u_0\|_{L^2_\uloc}\leq M$, and fix $\e\in (0,1)$.  If 
\[
0<T<T_\e:=\min(1, c\e^3 M^{-2}),
\]
then there exists a unique solution $u=u^\e$ to the integral form of \eqref{eq.NSEreg}
\EQ{
u(t)=e^{t\Delta}u_0 - \int_0^t e^{(t-s)\Delta} \mathbb P \nb \cdot (\mathcal J_\e (u) \otimes u \Phi_\e )(s)\,ds,
} 
satisfying 
\[
\esssup_{0<t<T}\sup_{x_0\in \R^3}\int_{B_1(x_0)} |u(x,t)|^2\,dx  +  \sup_{x_0\in \R^3}\int_0^T \!\int_{B_1(x_0)} |\nb u(x,t)|^2\,dx\,dt  \leq  C M^2,
\]
and $\lim_{t\to 0^+} \|u^\e(t)-u_0\|_{L^2(K)}=0$ for any compact subset $K$ of $\R^3$. 
Additionally, for $p^\e = (-\De)^{-1}\pd_i\pd_j (\mathcal J_\e (u) \otimes u \Phi_\e )$, we have
 $p^\e \in  L^{\I}(0,T;L^2(\R^3))$ and $u^\e$ and $p^\e$ solve \eqref{eq.NSEreg} in the sense of distributions.
\end{lemma}

The proof of Lemma \ref{lemma.KwTs} is contained in \cite{KwTs}. We next need an estimate for the solutions described in Lemma \ref{lemma.KwTs} for all scales. Note that this is just 
Lemma \ref{lemma.JiaSve} for the regularized system. 
The function $c_{x_0,r}^\e(t)$ is similar to $c_{x_0,r}(t)$ and will appear in the pressure decomposition formula \eqref{eq3.10} for $p^\e$.

\begin{lemma}\label{lemma.est.RegNSE}
Let $u_0\in L^2_\uloc$ with $\div u_0=0$  and fix $\e\in (0,1)$.  Assume for some $T\in (0,\I]$ that $u^\e$ and $p^\e$  satisfy all the conclusions of Lemma \ref{lemma.KwTs} on $\R^3\times (0,T)$. Then, for all $r>0$ we have
\begin{equation}\label{ineq.apriorilocalReg}
\esssup_{0\leq t \leq \sigma r^2\wedge T}\sup_{x_0\in \RR^3} \int_{B_r(x_0)}\frac {|u^\e|^2} 2 \,dx + \sup_{x_0\in \RR^3}\int_0^{\sigma r^2\wedge T}\!\!\int_{B_r(x_0)} |\nabla u^\e|^2\,dx\,dt <CA_0(r) ,
\end{equation}
and for some $c_{x_0,r}^\e(t) \in L^{3/2}_\loc([0,\sigma r^2\wedge T))$,
\begin{equation}\label{ineq.apriorilocalReg2}
\sup_{x_0\in \RR^3} \int_0^{\sigma r^2\wedge T}\!\!\int_{B_r(x_0) }\big( | u^\e|^3  +|p^\e-c^\e_{x_0,r}(t)|^{3/2}  \big)\,dx\,dt
 <C r^{\frac 12} A_{0}(r)^{\frac 32},
\end{equation}
where
\[
A_0(r)=rN^0_r= \sup_{x_0\in \R^3} \int_{B_r(x_0)} |u_0|^2 \,dx,
\]
and
\begin{equation*}
\si=\sigma(r) =c_0\, \min\big\{(N^0_r)^{-2} , 1  \big\},
\end{equation*}
for a small universal constant $c_0>0$. 
\end{lemma}

\begin{proof}
The proof is nearly identical to \cite[Proof of Lemma 2.2]{JiaSverak-minimal} and \cite[Appendix]{KMT}, the only difference being the estimates for the pressure and nonlinear terms.  These do not complicate things. Indeed, note that
\[
\|\Phi_\e \|_{L^\I}\leq C_\e,
\]
and
\[
\| \mathcal J_\e u \|_{L^p_\uloc}\leq C_\e\|  u \|_{L^p_\uloc},
\]
for $1\leq p <\I$. Using these facts, we obtain \cite[(2.8)]{JiaSverak-minimal}. To avoid redundancy, we omit further details.
\end{proof}

We next show global existence for the regularized system \eqref{eq.NSEreg} under the additional assumption 
\eqref{condition.forexistence}.

\begin{lemma}[Global existence for the regularized problem]\label{lemma.regularizedNSEexistence}
Assume $u_0\in L^2_\uloc$ satisfies \eqref{condition.forexistence} and is divergence free.  Then, there exists a solution $u^\e:\R^3\times (0,\I)\to \R^3$ to \eqref{eq.NSEreg} satisfying the a priori bounds in Lemma \ref{lemma.est.RegNSE} with $T = \I$.
\end{lemma}

\begin{proof}
We will iteratively construct a global-in-time solution.
For $n\in\N$, let 
\EQ{\label{Tn.def}
T_n = \inf_{j \ge n} \barT_j, \quad
\barT_n=\si(n)n^2,
}
where $\si$ is defined in Lemma \ref{lemma.est.RegNSE} (we are taking $r=n$).  The sequence $T_n$ is non-decreasing, $T_1>0$, and since $\barT_n\ge c_0(N_n^0/n)^{-2}$, $\lim_{n \to \infty}T_n = \infty$ by \eqref{condition.forexistence}.
 
\medskip

\emph{Step 1.} 
Let $M_1=\sqrt {CA_0(1)}$. By Lemma \ref{lemma.KwTs} with $M=M_1$, there exists a distributional solution $u^\e$ and pressure $p^\e$ to \eqref{eq.NSEreg} on $\R^3\times (0,T_\e)$ where $T_\e$ depends on $\e$ and $M_1$.  If $T_\e>T_1$ this step is over.  By Lemma \ref{lemma.est.RegNSE} with $T=T_\e$, $\|u(t_1)\|_{L^2_\uloc}< M_1$ for some $t_1\in (T_\e/2,T_\e)$. Hence, we can re-solve the regularized system \eqref{eq.NSEreg} with data $u(t_1)$ to obtain a second solution $\bar u$ on $\R^3\times (t_1 , t_1+T_\e) \subset (t_1,3T_\e/2)$.  By uniqueness in Lemma \ref{lemma.KwTs}, $u^\e=\bar u$ on $(t_1,T_\e)$. We can therefore extend $u^\e$ to $\R^3\times (0,3T_\e/2)$ by letting $u^\e = \bar u$ on $(T_\e,3T_\e/2)$.  If $3T_\e/2>T_1$ this step is done. Otherwise, note that $\|u(t_2)\|_{L^2_\uloc} <M_1$ for some $t_2\in (T_\e,3T_\e/2)$, and we can therefore repeat the extension argument to obtain a solution on a time scale extended by $T_\e/2$ units.  We can keep doing this, at each step extending the interval of existence by $T_\e/2$. Clearly, this will reach $T_1$ in finitely many steps.

\medskip

\emph{Step 2.} 
Let $M_2 =\sqrt{ C   A_0(2)}$.
If $T_2=T_1$ then we are done with this step.  Otherwise, 
we know by step 1 that a solution exists on $\R^3\times (0,T_1)$.  Let us redefine $T_\e$ to be the quantity from Lemma \ref{lemma.KwTs} with $M=M_2$
  (this is different than $T_\e$ from step 1).  By Lemma \ref{lemma.est.RegNSE}, we have $\|u(t)\|_{L^2_\uloc} \leq M_2$ for almost all $0<t<T_1$. So, there exists $t_2\in (T_1-T_\e/2,T_1)$ so that $\|u(t_2)\|_{L^2_\uloc} \leq M_2$.  Consequently, we can re-solve the regularized, localized Navier-Stokes equations using Lemma \ref{lemma.KwTs} starting at time $t_2$ to obtain a solution on $(t_2, t_2+T_\e )$. By uniqueness we can glue the new solution to the old solution to conclude that $u^\e$ and $p^\e$ are a solution on $\R^3\times (0,T_1+T_\e/2)$.  We can repeat this procedure finitely many times to obtain a solution $u^\e$ and pressure $p^\e$ on $\R^3\times (0,T_2)$.
  
\medskip

\emph{Step 3.} The procedure in Steps 1 and 2 can be iterated to obtain the following conclusion:
There exists a solution $u^\e$ and pressure $p^\e$ on $\R^3\times (0,  T_n)$ for all $n\in \N$.  Since $\{T_n\}$ is unbounded whenever \eqref{condition.forexistence} holds, $u^\e$ and $p^\e$ are a solution on $\R^3\times (0,\I)$. 
\end{proof}

\subsection{Global existence for the Navier-Stokes equations}\label{sec.existence.2} 
\begin{proof}[Proof of Theorem \ref{thrm.existence}]
Our argument mainly follows  \cite[\S3]{KwTs}, with a slight modification since our time scales must go to $\I$ (the basic elements of this argument were first written down in  \cite{LR} and later elaborated on in \cite{KiSe}).

We argue by induction.  For $\e>0$, let $u^\e$ and $\bar p^\e$ be the global-in-time solutions of the regularized system \eqref{eq.NSEreg} described in Lemma \ref{lemma.regularizedNSEexistence}.  Let $T_n$ be defined by \eqref{Tn.def}.
 Let $B_n$ denote the ball centered at the origin of radius $n$.  Then, Lemma \ref{lemma.regularizedNSEexistence} implies that $u^\e$ are uniformly bounded in the class from inequalities \cite[(4.1)-(4.4)]{KiSe} on $B_1\times [0,T_1]$.  Hence, there exists a sequence  $u^{1,k}$ (where the corresponding $\e$ are denoted by $\e_{1,k}$) that converges to a solution $u_1$ of \eqref{eq.NSE} on $B_1\times (0,T_1)$ in the following sense
\begin{align*}
&u^{1,k}\overset{\ast}\rightharpoonup u_1\quad \text{in }L^\I(0,T_1;L^2(B_1))
\\&u^{1,k}  \rightharpoonup u_1\quad \text{in }L^2(0,T_1;H^1(B_1))
\\&u^{1,k} \to u_1  \quad \text{in }L^3(0,T_1;L^3(B_1)) 
\\& \mathcal J_{\e_{1,k}}u^{1,k} \to u_1  \quad \text{in }L^3(0,T_1;L^3(B_1)).
\end{align*}

By Lemma \ref{lemma.regularizedNSEexistence}, all $u^{1,k}$ are also uniformly bounded on $B_n\times [0,T_n]$ for $n\in \N$, $n \ge2$ and, recursively, we can extract  subsequences $\{u^{n,k}\}_{k \in \N}$ from $\{u^{n-1,k}\}_{k \in \N}$ which converge to 
solution $u_n$ of \eqref{eq.NSE}  on $B_n\times (0,T_n)$ as $k \to \infty$ in the following sense
\begin{align*}
&u^{n,k} \overset{\ast}\rightharpoonup u_n\quad \text{in }L^\I(0,T_n;L^2(B_n))
\\&u^{n,k}  \rightharpoonup u_n\quad \text{in }L^2(0,T_n;H^1(B_n))
\\&u^{n,k} \to u_n \quad \text{in }L^3(0,T_n;L^3(B_n)) 
\\& \mathcal J_{\e_{{n,k}}} u^{n,k} \to u_n  \quad \text{in }L^3(0,T_n;L^3(B_n)).
\end{align*}
The difference here compared to \cite{KiSe} and \cite{KwTs} is that the time-scales depend on $n$.
Let $\tilde u_n$ be the extension by $0$ of $u_n$ to $\R^3\times (0,\I)$.
Note that, at each step, $\tilde u_n$ agrees with $\tilde u_{n-1}$ on $B_{n-1}\times (0, T_{n-1})$.  Let $u=\lim_{n\to \I}\tilde u_n$. Then, $u=u_n$ on $B_n\times (0,T_n)$ for every $n\in \N$.

Let $u^k = u^{k,k}$ on $B_k\times (0,T_k)$ and equal $0$ elsewhere. Let $\e_k$ denote the corresponding regularization parameter.  Then, for every fixed $n$ and as $k \to \infty$,
\EQ{\label{list.convergence}
&u^{k} \overset{\ast}\rightharpoonup u\quad \text{in }L^\I(0,T_n;L^2(B_n))
\\&u^{k}  \rightharpoonup u\quad \text{in }L^2(0,T_n;H^1(B_n))
\\&u^{k} \to u \quad \text{in }L^3(0,T_n;L^3(B_n)) 
\\& \mathcal J_{\e_{{k}}}u^{k} \to u \quad \text{in }L^3(0,T_n;L^3(B_n)).
}

Based on the uniform bounds for the approximates, we have that $u$ satisfies
\EQ{
&\sup_{0<t\leq T_n} \sup_{x_0\in \R^3} \int_{B_n(x_0)}|u(x,t)|^2\,dx 
\\&+ \sup_{x_0\in \R^3}\int_0^{T_n} \int_{B_n(x_0)} |\nb u(x,t)|^2\,dx\,dt \leq  C \sup_{x_0\in \R^3} \int_{B_n(x_0)} |u_0|^2\,dx.
}

To resolve the pressure, we follow \cite[\S3]{KwTs}. Let
\EQ{
p^k (x,t) =& -\frac 1 3 \mathcal J_{\e_k}(u^k  )\cdot u^k (x,t)\Phi_{\e_k} (x)+ \pv \int_{B_2} K_{ij}(x-y)  \mathcal J_{\e_k} (u^k_i  )\, u^k_j (y,t)\Phi_{\e_k}(y)\,dy  
\\&+ \pv \int_{B_2^c} (K_{ij}(x-y)-K_{ij}(-y))  \mathcal J_{\e_k}(u^k_i  )\, u^k_j (y,t)\Phi_{\e_k}(y)\,dy, 
}
which differs from the pressure $p^{\e_k}$ associated to $u^k=u^{\e_k}$ stated in Lemma \ref{lemma.KwTs} by a function of $t$ which is constant in $x$, and so $u^k$ with the above pressure $p^k$ is also a distributional solution to \eqref{eq.NSEreg} with $\e=\e_k$.

From the convergence properties of $u^k$, it follows that $p^k\to p$ in $L^{3/2}(0,T_n;L^{3/2} (B_n) )$ for all $n$ (this is \cite[(3.25)]{KwTs}) where $p$ is defined  as in \cite[(3.23)]{KwTs}, namely
\EQ{\label{p.def}
p(x,t)  = \lim_{n \to \infty} \bar p^n(x,t)
}
where $\bar p^n(x,t)$ is defined for $|x|<2^{n}$ by
\EQ{\label{p34.dec}
\bar p^n(x,t) 
=&-\frac 13 |u(x,t)|^2 
+\pv \int_{B_2}  K_{ij}(x-y) u_i u_j(y,t) dy + \bar p^n_3+\bar p^n_4,
}
with
\EQN{
\bar p^n_3(x,t) &=\pv \int_{B_{2^{n+1}}\setminus B_2}  (K_{ij}(x-y)-K_{ij}(-y)) u_iu_j(y,t)\,  dy ,
\\
\bar p^n_4(x,t) &=\int_{B_{2^{n+1}}^c}  (K_{ij}(x-y)-K_{ij}(-y)) u_iu_j(y,t) \, dy .
}
Note that $\bar p^n_4$ converges absolutely but $\bar p^n_3$ does not.
We have $\bar p^n_3,\, \bar p^n_4 \in L^{3/2}((0,T)\times B_{2^n})$ and
\[
\bar p^n_3+\bar p^n_4 = \bar p^{n+1}_3+\bar p^{n+1}_4 , \quad \text{in}\quad L^{3/2}((0,T)\times B_{2^n})
\]
Thus $\bar p^n(x,t)$ is independent of $n$ for $n> \log_2 |x|$.

Since above we followed \cite{KiSe} and \cite{KwTs}, we only established and used the local pressure expansion for scale $1$ and can only initially conclude that the local pressure expansion holds for scale $1$. We, however, need to establish this formula for all scales. The argument is actually the same but we include some details for convenience. 
Note that the local pressure expansion is valid for $p^k$ at all scales, that is, for any $T>0$, fixed $R>0$ and $x_0\in \R^3$, we have the following equality in $L^{3/2}(B_{2R}(x_0)\times (0,T))$,
\EQ{\label{eq3.10}
\hat p^k_{x_0,R}(x,t)   &:=p^k(x,t)-c_{x_0,R}^k(t)= -\Delta^{-1}\div \div [( \mathcal J_k u^k\otimes u^k \Phi_k)\chi_{4R} (x-x_0)]
\\&\quad - \int_{\R^3} (K(x-y) - K(x_0 -y)) ( \mathcal J_k u^k\otimes u^k\Phi_k)(y,t)(1-\chi_{4R}(y-x_0))\,dy,
}
where we are abusing notation by letting $\mathcal J_k=\mathcal J_{\e_k}$ and $\Phi_k = \Phi_{\e_k}$.  Similarly, let
\EQ{
 \hat p_{x_0,R}(x,t)&= -\Delta^{-1}\div \div [(  u\otimes u )\chi_{4R} (x-x_0)]
\\&\quad - \int_{\R^3} (K(x-y) - K(x_0 -y)) ( u\otimes u)(y,t)(1-\chi_{4R}(y-x_0))\,dy.
}
Fix $T>0$, $x_0\in \R^3$ and $R>0$.
Choose $n$ large enough that $B_{8R}(x_0)\times (0,T)\subset Q_n = B_n\times (0,T_n)$.  We claim that $\hat p^k_{x_0,R}(x,t)$ converges to $\hat p_{x_0,R}(x,t)$ in $L^{3/2}(B_{2R}(x_0)\times (0,T))$. 
If this is the case, by taking the limit of the weak form of \eqref{eq.NSEreg}, we can show that $(u,\hat p_{x_0,R})$ also satisfies \eqref{eq.NSE} in $B_{2R}(x_0)\times (0,T)$. Hence
$\nb p - \nb \hat p_{x_0,R}=0$, and we may define
\[
c_{x_0,R}(t):=p(x,t)- \hat p_{x_0,R}(x,t) 
\]
which is hence a function of $t$ in $L^{3/2}(0,T)$ that is independent of $x$. This gives  the desired local pressure expansion in $B_{2R}(x_0)\times (0,T)$.

To verify the claim we work term by term. Note that \cite[(3.26)]{KwTs} shows that
\[
\big\|    u_iu_j - (\mathcal J_k u_i^k) u_j^k \Phi_k \big\|_{L^{3/2} ( B_M\times [0,T]  ) } \to 0,
\] 
as $k\to \I$ for every $M>0$. For us, the same is true with $T$ replaced by $T_n$.  This implies
\EQ{
& -\Delta^{-1}\div \div [( \mathcal J_k u^k\otimes u^k \Phi_k)\chi_{4R} (x-x_0)] 
\\&\to -\Delta^{-1}\div \div [(  u\otimes u )\chi_{4R} (x-x_0)]  \text{ in } L^{3/2}(B_{2R}(x_0)\times (0,T_n)),
}
and
\EQ{
 &- \int_{|x|<M} (K(x-y) - K(x_0 -y)) ( \mathcal J_k u^k\otimes u^k\Phi_k)(y,t)(1-\chi_{4R}(y-x_0))\,dy 
\\&\to 
- \int_{|x|<M} (K(x-y) - K(x_0 -y)) ( u\otimes u)(y,t)(1-\chi_{4R}(y-x_0))\,dy,
}
in $L^{3/2}(B_{2R}(x_0)\times (0,T_n))$ for every $M>8R$.
For the far-field part, still assuming $M>8R$, we have
\EQ{
 & \bigg\| \int_{|x|\geq M} (K(x-y) - K(x_0 -y)) ( \mathcal J_k u^k\otimes u^k\Phi_k - u\otimes u)(y,t) \,dy 
 \bigg\|_{L^{3/2}(B_{2R}(x_0)\times (0,T_n))}
\\&\leq C(R,n,\|u_0\|_{L^2_\uloc}) M^{-1}.
}
This can be made arbitrarily small by taking $M$ large and noting $R$ and $n$ are fixed. Consequently, and since the other parts of the pressure converge, we conclude that $\hat p^k_{x_0,R}(x,t)$ converges to $\hat p_{x_0,R}(x,t)$ in $L^{3/2}(B_{2R}(x_0)\times (0,T_n))$, which leads to the desired local pressure expansion. Since $n$ was arbitrary, this gives the pressure formula for arbitrarily large times.

At this point we have established items 1.-3. from the definition of local energy solutions. 
The remaining items follow from the arguments  in \cite[pp. 156-158]{KiSe} and \cite[\S3]{KwTs}. This is because for any time $T_0$, we have the same convergences of $u^k$ and $p^k$ on $B_n\times T_0$ for all $n\in \N$ as in \cite{KiSe} and \cite{KwTs}. 
For convenience,  we briefly survey the details.
 
Fix $T_0$ and choose $n$ so that $T_n\geq T_0$. Then \eqref{list.convergence} holds for all $n$ with $T_n$ replaced by $T_0$.  Furthermore the estimates \cite[(4.1)-(4.4),(4.7),(4.9)]{KiSe} are valid up to a re-definition of $A$.  It follows from \cite[(4.7),(4.9)]{KiSe} that for every $n$,
\[
t\mapsto \int_{B_n} v\cdot w \,dx,
\]
is continuous on $[0,T_0]$ for every $w\in L^2(B_n)$ (alternatively, see \cite[(3.27)]{KwTs}). Since $T_0$ was arbitrary, we can extend this to all times. The local energy inequality follows from the local energy equality for $u^k$ and $p^k$, and \cite[(4.6)-(4.8),(4.10)]{KiSe} (we do not need \cite[(3.4)]{KiSe} since we did not regularize the initial data; see also \cite[(3.28)]{KwTs}). Convergence to the initial data in $L^2_\loc$ follows from \cite[(4.10), (4.12)]{KiSe}. This confirms that items 4.-6. from the definition of local energy solutions are satisfied and finishes the proof of Theorem \ref{thrm.existence}.
\end{proof}

\subsection{DSS local energy solutions for DSS data in $L^2_\uloc$}
\label{sec.existence.DSS}
 
We digress to reconsider a comment made in the introduction, in particular our claim that it is not difficult to show the discretely self-similar solutions constructed in \cite{BT5} are local energy solutions, when the initial data belong to $L^2_\uloc$.  We now explain how to do this.
In \cite{BT5}, we constructed a DSS solution pair $(u,p)$ to the Navier-Stokes equations as a limit of the DSS solutions $(u_k,p_k)$ with initial data $u_0^k\in L^{3,\I}\subset E^2$. The approximations satisfy the local pressure expansion and, consequently, are local energy solutions  (this follows from \cite{BT1} and \cite{KMT}). 
It is possible to show the local pressure expansion is inherited by $p$.
In particular, let $x_0\in \R^3$, $R>0$ and $T>0$. Let 
\EQ{
\hat p^k_{x_0,R} &= - \Delta^{-1}\div\div ( u_k\otimes u_k  \chi_{4R}(x-x_0))
\\& - \int_{\R^3} (  K(x-y)-K(x_0-y) ) (u_k\otimes u_k)(y,t) (1- \chi_{4R}(y-x_0))\,dy,
}
and 
\EQ{
\hat p_{x_0,R} &= - \Delta^{-1}\div\div ( u \otimes u  \chi_{4R}(x-x_0)) 
\\&- \int_{\R^3} (  K(x-y)-K(x_0-y) ) (u\otimes u)(y,t) (1- \chi_{4R}(y-x_0))\,dy,
}
where we are using notation from the proof of Theorem \ref{thrm.existence}. Since $u_0\in L^2_\uloc$ and since $(u_{k},p_k)$ are all local energy solutions, we have uniform estimates for $u_k$ by Lemma \ref{lemma.JiaSve}, provided $T$ is sufficiently small (depending on $u_0$). 
We also have $u\in L^\I L^2_\uloc$ and $\nb u$ satisfies 
\[
\sup_{x_0\in \R^3}\int_0^T	\int_{B_1(x_0)}	 |\nb u|^2\,dx\,dt <\I	.
\] 
Indeed, the convergence properties in \cite{BT5} and the argument in \cite[(3.18)-(3.20)]{KwTs}, imply the uniform bounds for $u_k$ are inherited by $u$.

We now know that $u^k,u\in L^\I(0,T ;L^2_\uloc)$ with uniform bounds and that $u_k$ converges to $u$ in $L^3( B_{8R}(x_0) \times (0,T)  )$. 
By the usual estimates (e.g. in the proof of Theorem \ref{thrm.existence}), it follows that 
\[
\hat p^k_{x_0,R} \to \hat p_{x_0,R},
\]
in $L^{3/2}(B_{2R}(x_0)\times (0,T))$.  
Since $(u_k,p_k)$ and $(u,p)$ solve the Navier-Stokes equations as distributions, the weak form of \eqref{eq.NSE} and the convergence properties in \cite{BT5} imply $\nb p^k_{x_0,R} = \nb p$ in $B_{2R}(x_0)\times (0,T)$ in the distributional sense. In particular, we have 
\[
\nb p^k \to \nb p
\]
in $\mathcal D' (  B_{2R}(x_0)\times (0,T))$ and 
\[
\nb p^k  = \nb p_{x_0,R}^k \to \nb p_{x_0,R},
\]
in $\mathcal D' ( B_{2R}(x_0)\times (0,T))$,
implying $\nb p = \nb p_{x_0,R}$ in $\mathcal D' ( B_{2R}(x_0)\times (0,T))$. We may thus define $c_{x_0,R} (t):= p(x,t)- \hat p_{x_0,R}$, which is a function in $L^{3/2}(0,T)$. Note that this argument was applied for some small $T$ (independent of $k$), but can be extended to all $T>0$ using discrete self-similarity.  This proves the solutions constructed in \cite{BT5} satisfy the local pressure expansion. 

To prove that the solution is a local energy solution, we must also prove some continuity in time, namely
\[
t\mapsto \int u(x,t) \cdot w (x)\,dx,
\]  
is continuous on $(0,\I)$ for any compactly supported $w\in L^2$.
This is known for $u_k$ by \cite{KMT} since these solutions have sufficient decay at spatial infinity. This follows for $u$ in the usual way -- in particular see the argument preceding \cite[(3.27)]{KwTs}.

\section{Uniqueness}\label{sec.unique}
In this section we prove Theorems \ref{thrm.uniquenessa}, \ref{thrm.uniquenessb} and  \ref{thrm.uniqueness2}.  Theorem \ref{thrm.uniquenessa} will be proven following the theme of Jia \cite[Proof of Theorem 3.1]{Jia-uniqueness}. 
There are two main differences in our approach:
First, when $u_0\in L^{3,\I}$, we have \[
\| e^{t\Delta}u_0\|_{L^4_\uloc}\lesssim t^{-1/8}\|u_0\|_{L^{3,\I}}.
\]
Interestingly, this breaks down when we replace $L^{3,\I}$ by $M^{2,1}$, as is shown in Example \ref{counterexample}. Due to this we need to modify Jia's argument.  The modification is similar to the setup in \cite{LR-Morrey}.
Second, the integral formula for mild solutions has to be checked explicitly since $M^{2,1}$ does not embed in $E^2$ while $L^{3,\I}$ does, see Lemma \ref{lemma62}.  Membership in $E^2$ is enough to guarantee that a local energy solution is a mild solution; 
see \cite[\S8]{KMT}.

 \begin{proof}[Proof of Theorem \ref{thrm.uniquenessa}]
By Kato \cite[Lemmas 2.2 and 4.2]{Kato}, we have 
\EQ{\label{0508-a}
\norm{e^{t \De}u_0}_{M^{2,1}}  &\le     \norm{u_0}_{M^{2,1}},
}
\EQ{\label{0508-b}
t^{1/2}\norm{e^{t\De} \mathbb P \nb\cdot F}_{ M^{2,1}} +
t^{3/4}\norm{e^{t\De} \mathbb P \nb\cdot F}_{ M^{4,1}} &\le C \norm{ F}_{M^{2,1}},
}
where  $\mathbb P$ is the Helmholtz projection in $\R^3$, which is bounded in Morrey spaces by \cite[Lemma 4.2]{Kato}.
Also note
\EQ{\label{0508-b2}
\norm{uv}_{M^{4,1}} \lec \norm{u}_{L^\infty}\cdot \norm{v}_{M^{4,1}},\quad
\norm{uv}_{M^{2,1}} \lec \norm{u}_{M^{4,1}}\cdot \norm{v}_{M^{4,1}}.
}

Let $u_0\in M^{2,1}$ be as in the statement of Theorem \ref{thrm.uniquenessa} with $\e = \norm{u_0}_{M^{2,1}}$ sufficiently small. 
Note $\sup_{0<r<\infty} N_r^0 \le C\norm{u_0}_{M^{2,1}}^2 < 1$. Thus $\si(r) = c_0$ for all $r>0$, where $\si(r)$ and $c_0$ are defined in  \eqref{def.sigma}.
Let $u\in \mathcal N(u_0)$.
By the third part of Theorem \ref{thrm.littleo}, we have
\EQ{\label{0508-c} 
\norm{u(t)}_{L^\infty} \le C \e\, t^{-1/2},\quad (0<t<\I).
}
Fix $t\in (0,\I)$.
Let $r_t = \sqrt{t/c_0}$.  Using \eqref{ineq.apriorilocal} we have 
\EQ{\label{ineq.Morrey1}
\sup_{r>r_t, x\in \R^3} \frac 1r \int_{B(x,r)} |u(t)|^2 <  C\e^2.
}
For $r< r_t$, using \eqref{0508-c} and the above at $r=r_t$,
\EQ{\label{ineq.Morrey2}
\frac 1r \int_{B_r(x)} |u(t)|^2  &= \frac 1r \bke{ \int_{B_r(x)} |u(t)|^2}^{1/3} \bke{ \int_{B_r(x)} |u(t)|^2}^{2/3} 
\\
&\lec \norm{u(t)}_{L^\infty} ^{2/3} \bke{\int_{B(x,\, {r_t})} |u(t)|^2}^{2/3} 
\\
&\lec C(\e) t^{-1/3} r_t ^{2/3}  = C(\e).
}
Because $t$ was arbitrary, we have shown that,
\EQ{\label{0508-d}
\sup_{0<t<\infty} \norm{u(t)}_{M^{2,1}} \le C( \e).
}
Hence
\EQ{\label{0508-e}
 \norm{u(t)}_{M^{4,1}} \le \norm{u(t)}_{L^\infty}^{1/2} \norm{u(t)}_{M^{2,1}} ^{1/2} \le C( \e) t^{-1/4},
 \quad (0<t<\infty).
}

We now show that $u$ is a mild solution, that is, $u$ satisfies the integral form of the Navier-Stokes equations
\EQ{\label{INS}
	u(x,t)=e^{t\Delta}u_0(x) -\int_0^t e^{(t-s)\Delta}\mathbb P \nabla\cdot (u\otimes u) (s)\,ds.
}
Let
\EQ{\label{tdu.def}
\td u(x,t) = e^{t\Delta}u_0(x) -\int_0^t e^{(t-s)\Delta}\mathbb P \nabla\cdot (u\otimes u) (s)\,ds.
}
  
By \eqref{0508-a},
\EQ{\label{ineq.1.8.19.a}
\| e^{t\Delta}u_0\|_{M^{2,1}}\leq \| u_0\|_{M^{2,1}}.
}
By \eqref{0508-b}, \eqref{0508-c} and \eqref{0508-d},
\EQ{\label{ineq.1.8.19.b}
\bigg\| \int_0^t  e^{(t-s)\Delta} \mathbb P \nabla\cdot (u\otimes u) (s) \,ds\bigg\|_{M^{2,1}}
&\leq \int_0^t \frac C {(t-s)^{1/2}} \| u\otimes u (s)\|_{M^{2,1}}\,ds
\\&\leq  \int_0^t \frac C {(t-s)^{1/2}} \|u(s)\|_{L^\I} \|  u (s)\|_{M^{2,1}}\,ds
\\&\leq   \int_0^t \frac {C(\e)} {(t-s)^{1/2} s^{1/2}}  \,ds
= C(\e).
}
Thus
\EQ{
\sup_{0<t<\I}\|\td u(t)\|_{M^{2,1}} \leq C(\e).
}

Let $U=u-\td u$.  Then, $U\in L^\I M^{2,1}$ and, therefore, so is $U_\e= \eta_\e* U$. As in \cite[after (8.14)]{KMT},  $\om_\e=\curl U_\e$ is a bounded solution to the heat equation with zero initial data and is therefore equivalently $0$.  Hence, $U_\e$ is curl free and divergence free, implying it is harmonic.  Thus, for any $x_0\in \R^3$ and $t>0$, we have for all $r>0$ that
\EQ{
U_\e(x_0,t) = \frac C {r^3}\int_{B(x_0,r)} U_\e (y,t)\,dy.
}
So 
\EQ{
|U_\e(x_0,t)|\leq \inf_{r>0} \frac C {r}\bigg(\frac 1 r \int_{B(x_0,r)} |U_\e(y,t)|^2\,dy \bigg)^{1/2}.
}
The right hand side of the above inequality is zero because $U_\e (t)\in M^{2,1}$.  Therefore, $U_\e \equiv 0 $ for all $\e>0$ and, therefore, $U=0$.  It follows that $u$ is a mild solution.

Assume $v\in\mathcal N (u_0)$ also.  Then, $v$ also satisfies an integral formula.
Let $w=u-v$.    Then,
\[ 
w(\cdot, t)=-\int_0^t e^{(t-\tau)\Delta}\mathbb P \nabla\cdot( w\otimes w + v\otimes w+w\otimes v   )(\cdot,\tau)\,d\tau.
\]
Let $\al(t)=\esssup_{0\leq s\leq t}s^{1/4}\|w(s)\|_{M^{4,1}}$.
Using \eqref{0508-b}, \eqref{0508-b2} and \eqref{0508-e}, we have
$\al(t) \le C(\e)$ and,
 for $0<s<t$
\EQN{
\|w(s)\|_{M^{4,1}}&\leq \int_0^s \frac C {(s-\tau)^{\frac 34}} \| w\otimes w + v\otimes w+w\otimes v  \|_{M^{2,1}}(\tau)\,d\tau 
\\
&\leq \int_0^s \frac C {(s-\tau)^{\frac 34}}\bke{( \| u (\tau) \|_{M^{4,1}} + \| v(\tau) \|_{M^{4,1} } ) \| w(\tau)\|_{M^{4,1} }   }\,d\tau 
\\
&\leq \int_0^s \frac {C(\e)} {(s-\tau)^{\frac 34} }   \tau^{-1/2} \al(t)\,d\tau
\\
& = C(\e) s^{-1/4} \al(t).  
}
Taking esssup in $s\in(0,t)$, we get
\[
\al(t) \le  C(\e) \al(t).
\]
If we take $\e>0$ sufficiently small such that $C(\e) < 1$, we get $\al(t)=0$. Therefore, $u=v$.  This concludes the case when $\|u_0\|_{M^{2,1}}$ is small.
\end{proof}

 \begin{proof}[Proof of Theorem \ref{thrm.uniquenessb}]
Assume $\limsup_{R\to 0}N^0_R <\e$ for some $\e>0$,
 and either $u_0$ satisfies \eqref{condition.forexistence} or $u_0\in E^2$.  
We will prove that if $\e$ is sufficiently small, then there exists $T>0$ so that $u=v$ on $\R^3\times (0,T)$ as distributions.  Let $R_0$ satisfy   $\sup_{R<R_0}{N^0_R} <\e$.    By Theorem \ref{thrm.littleo} we have for $T=c_0  R_0^2$,
\[
 t^{1/2}\|u(\cdot, t)\|_{L^\I}\leq C(\e),\quad (0<t\le T).
\]
Now, using the estimates \eqref{ineq.Morrey1} (for $r=\sqrt T$ only) and \eqref{ineq.Morrey2}  (for $r<\sqrt T$) we have for all $t\in (0,T)$ and $r\in (0,T^{1/2})$ that 
\[
\frac 1 r \int_{B_r(x)} |u(x,t)|^2\,dx \leq C(\e),
\] 
implying $\|u(t)\|_{M^{2,1}_{<T^{1/2}}} \leq C(\e)$ for all $t<T$. Also note that 
\[
\|u(t)\|_{M^{4,1}_{<T^{1/2}}} \leq \|u (t)\|_{L^\I}^{1/2} \|u(t)\|_{M^{2,1}_{<T^{1/2}}}^{1/2}
\leq C(\e) t^{-1/4}.
\]

In the next step we check that $u$ satisfies the integral formula \eqref{INS} on $\R^3\times (0,T)$. 
If $u_0\in E^2$, then this follows from \cite[\S8]{KMT}. On the other hand, assume that $u_0$ satisfies \eqref{condition.forexistence}.  
By \eqref{ineq.apriorilocal},
\EQN{
\esssup_{0\leq t\leq \si(r) r^2} \sup_{x_0\in \R^3} \int_{B_r(x_0)} |u|^2\,dx  < C  A_0(r), \quad A_0(r)=\sup_{x_0\in \R^3}\int_{B_r(x_0)}  |u_0|^2\,dx,
}
where $\si(r) = c_0 \min \{  r^2 (A_0(r))^{-2}  ,1 \}$.
Since $u_0$ satisfies \eqref{condition.forexistence}, we have $\si(r) r^2 \to \I$ as $r\to \I$.  So, there exists $\bar R$ so that, for all $R>\bar R$, $\si(R)R^2 > T$. We conclude for any $r>0$
\EQ{\label{eq4.16}
\esssup_{0\leq t\leq T} \sup_{x_0\in \R^3 } \int_{B_r(x_0)} |u(x,t)|^2\,dx \leq f(r),
}
where   
$f(r) =C A_0(r)+ C A_0(\bar R)$.

Let $\td u$ be defined by \eqref{tdu.def}. Denote 
\[
\norm{w}_{L^q_{\uloc,r}} = \sup _{x \in \R^3}\norm{w}_{L^q(B(x,r))}.
\]
Recall Maekawa-Terasawa \cite[(1.8),(1.10)]{MaTe}, for $1 \le q \le p \le \infty$,
\EQ{\label{MTa}
\norm{e^{t\Delta}u_0}_{L^p_{\uloc,r}} \le C  \min(\sqrt t,r)^{\frac 3p -\frac 3q}  \norm{u_0}_{L^q_{\uloc,r}},
}
\EQ{\label{MTb}
\norm{e^{t\Delta}\mathbb P \nb \cdot F}_{L^p_{\uloc,r}} \le  Ct^{-1/2} \min(\sqrt t,r)^{\frac 3p -\frac 3q} \norm{ F}_{L^q_{\uloc,r}}.
}
The same computations in \eqref{ineq.1.8.19.a}--\eqref{ineq.1.8.19.b} with $M^{2,1}$ replaced by $L^2_{\uloc,r}$ and using \eqref{MTa}-\eqref{MTb} instead of \eqref{0508-a}--\eqref{0508-b} give
\[
\esssup_{0\leq t\leq T} \sup_{x_0\in \R^3 } \int_{B_r(x_0)} |\td u(x,t)|^2 \,dx\le C A_0(r) +C(\e) \sup_{0<s<T}\sup_{x_0\in \R^3 } \int_{B_r(x_0)} | u(x,s)|^2 \,dx\le
Cf(r).
\]
Thus $U=u - \td u$ satisfies
\EQ{\label{U.bound}
\esssup_{0\leq t\leq T} \sup_{x_0\in \R^3 } \int_{B_r(x_0)} |U(x,t)|^2\,dx \leq C f(r),\quad \forall r>0.
}
The same argument in \cite[\S8]{KMT} shows that mollified $U_\e$ is harmonic in $x$ and for fixed $t \in (0,T)$
we have
\EQ{\label{Ue.vanishing}
|U_\e(x,t)| \le C \bke{\frac 1{r^3} \int_{B(x,r)} |U_\e(y,t)|^2 \,dy}^{1/2} \le C \bke{\frac {f(r)}{r^3}}^{1/2}\to 0 \quad\text{as}\quad r\to \infty.
}
This shows $U_\e(t)=0$ for $t<T$, for all $\e>0$. Hence $U=0$ and $u = \td u$.
 
At this stage we have shown that any local energy solution with data satisfying the assumptions of  Theorem \ref{thrm.uniquenessb} is a mild solution.

We continue similarly to the proof of Theorem \ref{thrm.uniquenessa}. Let $w=u-v$ where $u$ and $v$ are local energy solutions with the same data satisfying the assumptions of  Theorem \ref{thrm.uniquenessb}.  

Using \eqref{MTb} we have for $s<T$ and $R=R_0$ that
\EQ{
\|w(s)\|_{L^4_{\uloc,R}} 
&\leq \int_0^s \frac C {(s-\tau)^{\frac 12}} \| w\otimes w + v\otimes w+w\otimes v  \|_{L^4_{\uloc,R}}(\tau)\,d\tau 
\\
&\leq \int_0^s \frac C {(s-\tau)^{\frac 12}  \tau^{\frac 14}}  (\|v\|_\I +\|w\|_\I)  \tau^{1/4}\| w  \|_{L^4_{\uloc,R}}(\tau)\,d\tau
\\
&\leq CC(\e)\int_0^s \frac 1 {(s-\tau)^{\frac 12}  \tau^{\frac 34}}    \tau^{1/4}\| w  \|_{L^4_{\uloc,R}}(\tau)\,d\tau .
}  
Let $\al(s)=\sup_{0<\tau\leq s}\tau^{1/4} \|w(\tau)\|_{L^4_{\uloc,R}}$.
Then,
\EQ{
\|w(s)\|_{L^4_{\uloc,R}} &\leq CC(\e)\al(s)  \int_0^s  \frac 1 {(s-\tau)^{\frac 12}  \tau^{\frac 34}}  \,d\tau \leq CC(\e) \al(s) s^{-1/4}.
}
Taking the  
essential supremum over $s\in (0,T)$ gives
\[
\al(T) \le C C(\e) \al(T),
\]
and, taking $\e$ sufficiently small we obtain uniqueness on $(0,T)$.  
\end{proof}
 
\begin{remark} 
 Uniqueness in Theorem \ref{thrm.uniquenessb} cannot be extended beyond $T$, since the smallness of $t^{1/2}\|u(\cdot, t)\|_{L^\I}$ for $t>T$ is unknown. 
\end{remark}

\begin{remark} 
 To get $U_\e=0$ by \eqref{Ue.vanishing}, we only need $f(r) = o(r^3)$ as $r \to \infty$. Assumption
\eqref{condition.forexistence} is needed to get the a priori bound  \eqref{U.bound} for all $r$ and a \emph{fixed} $T$.
 If \eqref{condition.forexistence} is replaced by a weaker condition $\lim_{r \to \infty}r^{-2}A_0(r) = \de>0$, then we can show a priori bound of $u$ up to time $T'=\liminf_{r\to \infty}c(r)r^2>0$, and we can still get $U_\e=0$ for $0<t<\min(T,T')$ by \eqref{Ue.vanishing}.
\end{remark}

We now prove Theorem \ref{thrm.uniqueness2}. 

\begin{proof}[Proof of Theorem \ref{thrm.uniqueness2}]
Fix $0<T<\I$. Note that since $u_0\in E^2$, $u(t)\in E^2$ for every $t$ (see \cite{KiSe,KMT}).
For $t\in [0,T]$, let $\bar r(t)\leq {1}$ be the largest scale so that
\[
\sup_{r<\bar r(t), x_0\in \R^3}\frac 1 r \int_{B_r(x_0)} |u(x,t)|^2 \,dx <\e_2.
\]
If $\{\bar r(t)\}_{t\in [0,T]}$ is bounded away from $0$, say by $r_0$, then we are done.  Indeed, applying Theorem \ref{thrm.uniquenessb}
at time $t=0$, we obtain uniqueness on $[0,cr_0^2]$.  Then, applying Theorem \ref{thrm.uniquenessb} at   time $t = cr_0^2$, we obtain uniqueness up to time $2cr_0^2$. This argument is iterated a finite number of times to obtain uniqueness on $[0,T]$.  If this can be done for any $T>0$, then we have $u=v$.

We must prove $\{\bar r(t)\}_{t\in [0,T]}$ is bounded away from $0$.  Suppose there exists a time $t_* \in [0,T]$ and a sequence $0\leq t_n\to t_*$ so that $\{\bar r(t_n) \}$ is not bounded away from zero.  We may assume   $\bar r(t_n)$ decreases to zero.  Since $u(t_*)\in m_{\e}^{2,1}$, there exists $\td r$ so that
\[
\sup_{r<\td r, x_0\in \R^3}\frac 1 r \int_{B_r(x_0)} |u(x,t_*)|^2 \,dx \leq \e. 
\]
By continuity we have
\[
 \| u(t_n) - u(t_*)\|_{M^{2,1}_{<\td r  }} \le \| u(t_n)-u(t_*)\|_{m^{2,1}} <\sqrt{\e_2}-\sqrt{\e},
\]
for $n$ sufficiently large. 
But then 
\[
\sup_{r<\td r, x_0\in \R^3}\frac 1 r \int_{B_r(x_0)} |u(x,t_n)|^2 \,dx=\| u(t_n)\|_{M^{2,1}_{<\td r  }}^2  \leq
 \bke{\| u(t_n) - u(t_*)\|_{M^{2,1}_{<\td r  }}  + \|   u(t_*)\|_{M^{2,1}_{<\td r  }}}^2 < \e_2,
\]
implying $\bar r(t_n)\geq \td r$ for $n$ sufficiently large. This is a contradiction.  Thus, $\{\bar r(t)\}$ is bounded away from zero and uniqueness follows.
\end{proof}

The proof of Theorem 1.8 can be modified to show $\bar r(t)$ is lower semicontinuous.


\section{Proofs of corollaries }\label{sec.proofs2}

We begin this section with a lemma concerning the relationships between the function spaces introduced in Section \ref{sec.intro}. We then  prove the corollaries from Section \ref{sec.intro}. 
 
\begin{lemma}\label{lemma.morreyspacedecay}
If $u_0\in \td M^{2,1}$, then $\lim_{R\to\I}R^{-1}A_0(R)=0$ and $\lim_{R\to 0}R^{-1}A_0(R)=0$.  
If $u_0\in  \td m^{2,1}$, then  $\lim_{R\to 0}R^{-1}A_0(R)=0$.
\end{lemma}
 
\begin{proof}
Assume $\phi_k\to u_0$ in $M^{2,1}$ where $\{\phi_k\}\subset C_c^\I$.  Let $\e>0$ be given.  Then, there exists $k$ so that 
\[
\sup_{x_0\in \R^3; R>0}\frac 1 R \int_{B_R(x_0)} |u_0-\phi_k|^2\,dx <\e/4.
\]
Furthermore, there exists $R_k$ so that, for all $R<R_k$, 
\[
\sup_{x_0\in \R^3}\frac 1 R \int_{B_R(x_0)}|\phi_k|^2\,dx <\e/4.
\]
It follows that, for all $R<R_k$,
\[
\sup_{x_0\in \R^3} \frac 1 R \int_{B_R(x_0)}|u_0|^2\,dx <\e.
\]
On the other hand there exists $R^k>0$ so that for all $R>R^k$,
\[
\sup_{x_0\in \R^3}\frac 1 R \int_{B_R(x_0)}|\phi_k|^2\,dx <\e/4.
\]
Therefore, for $R>R^k$,
\[
\sup_{x_0\in \R^3} \frac 1 R \int_{B_R(x_0)}|u_0|^2\,dx <\e.
\]

The proof is similar for $\td m^{2,1}$.
\end{proof}

\begin{remark}\label{lemma.morreyspacedecay-rmk}
The converse statement is not true. A function $u_0$ in $m^{2,1}$ satisfying the vanishing property $\lim_{R\to 0}R^{-1}A_0(R)=0$ (i.e., $u_0 \in m^{2,1}_0$) may not be in $\td m^{2,1}$. For example, 
\[
u_0(x) = \sum_{k \in \Z} \zeta(x-ke_1)
\]
where $\zeta(x)\in C^\infty_c$ is supported in $|x|<1/4$. This function is actually in $M^{2,1}$, and is similar to that in \eqref{eq1.13}. Another example is $u_0(x) = \psi(x_2,x_3)$ where $\psi \in C^\infty_c(\R^2)$.
\end{remark}

 \medskip

We are now ready to prove the corollaries stated in Section \ref{sec.intro}.

 \begin{proof}[Proof of Corollary \ref{cor.morrey}]   Assume $u_0\in L^2_{\uloc,\si}$ and $u$ is a local energy solution to \eqref{eq.NSE} with initial data $u_0$.  

\bigskip 

\noindent 1.  If $u_0\in M^{2,r}$ for $0\leq r<1$, then $\lim_{R\to\I}R^{-1}A_0(R)=0$ and we can apply Theorem \ref{thrm.littleo}.1 to get the desired result.  On the other hand, if  $u_0\in M^{2,r}$ for $1<r\leq 3$, then $\lim_{R\to0^+}R^{-1}A_0(R)=0$, then Theorem \ref{thrm.littleo}.2 yields the desired result.
\bigskip 

\noindent 2. If $u_0\in \overline{C_c^\I}^{M^{2,1}}$, then, by Lemma \ref{lemma.morreyspacedecay}, $\lim_{R\to\I}R^{-1}A_0(R)=0$ and $\lim_{R\to 0}R^{-1}A_0(R)=0$ and we can apply Theorem \ref{thrm.littleo}.1 and \ref{thrm.littleo}.2 to get the desired result.  For the secondary conclusion, assume $u_0\in L^{3,q}$ where  $1\leq q< \I$.  Since $C_c^\I$ is dense in $L^{3,q}$ when $1\leq q<\I$, there exists a sequence $\{\phi_n\}\subset C_c^\I$ so that $\phi_n\to u_0$ in $L^{3,q}$.   By the continuous embedding of $L^{3,q}$ into $M^{2,1}$, we have $\phi_n\to u_0$ in $M^{2,1}$ also. 

\bigskip 
\noindent 3. 
If $u_0\in L^2$ then obviously
\[
\lim_{\R\to \I} \frac 1 R \int_{B_R(x_0)}|u_0|^2\,dx =0.
\]
For the other case assume $2<p<\I$ and $1\leq q\leq \I$. Using Lemma \ref{weakLq-into-Morrey} from the appendix we have 
\[
\int_{B_R(x_0)} |u_0|^2\,dx \leq C \|u_0\|_{L^{p,q}}^2 |B_R(x_0)|^{1-2/p}.
\]
Hence,
\[
\sup_{x_0\in \R^3}  \frac 1 {R} \int_{B_R(x_0)} |u_0|^2\,dx \leq CR^{2-6/p}\|u_0\|_{L^{p,q}}^2.
\]
If $2<p<3$, this vanishes as $R\to \I$ and if $3<p<\I$, this vanishes as $R\to 0^+$.  These correspond to cases from Theorem \ref{thrm.littleo} and the corollary follows.
 \end{proof}

\begin{proof}[Proof of Corollaries \ref{cor.DSS1} and \ref{cor.DSS2}]
Assume $u_0\in L^2_\uloc$ is divergence free and $\la$-DSS for some $\la>1$.   Lemma \ref{lemma.dss.equiv} implies $u_0\in M^{2,1}$ and 
\EQ{
 \norm{u_0}_{M^{2,1}}\le \sqrt \la \, \|u_0\|_{L^2_{\uloc}}.
}
If $\|u_0\|_{L^2_{\uloc}}\leq \e_1/\la$, then $\sup_{0<R<\I}R^{-1}A_0(R)<  \e_1$.  Applying Theorem \ref{thrm.littleo}.3 completes the proof of Corollary \ref{cor.DSS1}.  On the other hand, if $\|u_0\|_{L^2_{\uloc}}\leq \e_2/\la $, then applying Theorem \ref{thrm.uniquenessa} completes the proof of Corollary \ref{cor.DSS2}
\end{proof}

 \begin{proof}[Proof of Corollary \ref{cor.uniqueness}]
Assume $u_0\in E^3$.   Then, $u_0\in E^2$.
Let $\e>0$ be given.  Then, there exists $R_0$ so that  \[   \sup_{|x_0|\geq R_0 }\int_{B_1(x_0)}|u_0(x)|^3\,dx < \e.\] 
For $|x_0|\leq R_0$, there exists $\ga \in(0,1]$ so that 
\[
\sup_{|x_0|\leq R_0;0<r\leq \ga} \int_{B_r (x_0) }|u_0(x)|^3\,dx<\e,
\]
for all $r\leq\ga$.
Using H\"older's inequality, it follows that 
\[
\sup_{x_0\in \R^3; 0<r \leq  \ga} \frac 1 {r}\int_{B_{r}(x_0)}  |u_0 (x)|^2\,dx <|B_1|^{1/3} \e^{\frac 2 3}.
\]
Hence, by Theorem \ref{thrm.uniquenessb}, any local energy solution with initial data $u_0$ will be unique in the local energy class, at least up to some positive time.
\end{proof}

\begin{proof}[Proof of Corollary \ref{cor.uniqueness2}] 

Note that the norm on $\td m^{2,1}$ is just the $m^{2,1}$ norm. Hence, if $u\in C([0,\I);\td m^{2,1})$ then $u(t)\in \td m^{2,1}\subset m_\e^{2,1}$ for every $\e>0$ and $u\in C([0,\I);m^{2,1})$. By Lemma \ref{lemma.morreyspacedecay} we thus have $u\in C([0,\I);m_{\e_2/2}^{2,1})$. We also have $u_0 \in \td m^{2,1} \subset E^2$.
To apply Theorem \ref{thrm.uniqueness2}, it thus suffices to show $u\in \mathcal N(u_0)$.
All items from Definition \ref{def:localLeray} are immediate except for the local energy inequality and the local pressure expansion.

For the local energy inequality, note that
$\norm{f}_{L^4} \lec \norm{f}_{M^{2,1}}+ \norm{\nb f}_{L^2}$. The same inequality for $f_\la(x)=f(\la x)$ gives
$\la^{-3/4}\norm{f}_{L^4} \lec \la^{-1}\norm{f}_{M^{2,1}}+ \la^{1-3/2} \norm{\nb f}_{L^2}$. Optimizing in $\la$, 
\EQ{
\norm{f}_{L^4} \lec\norm{f}_{M^{2,1}}^{1/2} \norm{\nb f}_{L^2}^{1/2}.
}
Letting $\psi_{x_0} = \psi(\cdot - x_0)$,  where $\psi=1$ on $B_1(0)$, be in $C_c^\I$,   nonnegative  and  supported on $B_2(0)$, we have by the preceding inequality that
\[
\|u  \psi_{x_0}\|_{L^4_x} \lec
 \|u \psi_{x_0} \|_{M^{2,1}}^{1/2}  \norm{u \psi_{x_0}}_{H^1}^{1/2} 
\lec 
\|\psi \|_{W^{1,\infty}}^{1/2} \|u \psi_{x_0} \|_{M^{2,1}}^{1/2}   \norm{u }_{H^1(B_2(x_0))}^{1/2} .
\]
Now, for $\bar x_0\in \R^3$ and $r\geq 1$, 
\EQN{
\frac 1 {r} \int_{B_r(\bar x_0) \cap B_{2}(x_0)} |u(x,t) \psi_{x_0}|^2\,dx &\leq  \|\psi\|_{L^\I}^2  \int_{B_2(x_0)} |u(x,t)|^2\,dx \lesssim  \|\psi\|_{L^\I}^2  \|u\|_{m^{2,1}}^2,
}
while if $r<1$ we have 
\EQN{
\frac 1 {r} \int_{B_r(\bar x_0) \cap B_{2}(x_0)} |u(x,t) \psi_{x_0}|^2\,dx \leq  \|\psi\|_{L^\I}^2 \| u\|_{m^{2,1}}^2.
}
Hence
\EQ{
\|u  \psi_{x_0}\|_{L^4} \lec \|\psi \|_{W^{1,\infty}} \|u   \|_{m^{2,1}}^{1/2}   \norm{u }_{H^1(B_2(x_0))}^{1/2}.
}
Thus  
\[
\|u  \psi_{x_0}\|_{L^4_tL^4_x} \lec \|\psi \|_{W^{1,\infty}} \|u  \|_{L^\infty m^{2,1}}^{1/2} \norm{u }_{L^2 H^1(B_2(x_0))}^{1/2}.
\]
It follows from our assumptions that
\[
\sup_{x_0\in \R^3}\int_0^T \int_{B_1(x_0)} |u|^4\,dx \,dt <\I.
\]
This guarantees that $u$ satisfies the local energy inequality.\footnote{That  $u\in L^4_\uloc( \R^3\times [0,T] )$  can also be proven using the  embedding $m^{2,1}\subset   B^{-1}_{\I,\I}$ and \cite[(2)]{LR-Besov}.}

Concerning the local pressure expansion, let $\e>0$ be given and fix $R>0$.  For each $t\in [0,T]$, there exists $\phi_t\in C_c^\I$ so that  $\|u(t)-\phi_t\|_{m^{2,1}}<\e/(2R^3)$. By continuity, there exists an open interval $I_t$ containing $t$ so that $\|u(s)-\phi_t\|_{m^{2,1}}<\e/R^3$ for all $s\in I_t\cap [0,T]$.   Since $[0,T]$ is compact, we may cover $[0,T]$ using finitely many $I_{t_i}$, $1\leq i\leq k$. Then, restricting $I_{t_i}$ to a disjoint cover of $[0,T]$ by intervals $\td I_{t_i}$, we let $\phi (x,t) = \phi_{t_i}(x)\chi_{\td I_{t_i}}(t)$.  We then have
\[
\int_0^{R^2\wedge T} \int_{B_R(x_0)} |u(x,t)-\phi(x,t)|^2\,dx\,dt < \e .
\] 
Now, the spatial support of $\phi$ is contained in a compact set $K$.  So, for $|x_0|$ sufficiently large and $x\in B_R(x_0)$, $\phi(x,t)=0$ for all $t$.  Hence
\[
\int_0^{R^2\wedge T} \int_{B_R(x_0)} |u(x,t)|^2\,dx\,dt < {\e},
\]
for $|x_0|$ sufficiently large. This is the sufficient condition given in \cite{JiaSverak-minimal} for the validity of the local pressure expansion. 

We have thus shown $u\in \mathcal N(u_0)$ and can appeal to Theorem \ref{thrm.uniqueness2} to complete the proof.

\end{proof}

\section{Appendix: Relations between function spaces}\label{sec.appendix}

For clarity we include several helpful facts about the function spaces considered in this paper. These facts are known, but we include proofs for convenience.

\begin{lemma}\label{lemma.dss.equiv}
Assume $u_0$ is $\la$-DSS.  Then, $u_0\in M^{2,1}$ if and only if $u_0\in L^2_\uloc$ and 
\EQ{\label{DSSL2uloc}
\|u_0\|_{L^2_\uloc} \leq \|u_0\|_{M^{2,1}}\leq \sqrt \la \|u_0\|_{L^2_\uloc}.
}
\end{lemma}
\begin{proof}Assume  $u_0$ is $\la$-DSS. If $u_0\in M^{2,1}$ then we clearly have $u_0\in L^2_\uloc$ and 
\[
\|u_0\|_{L^2_{\uloc}}\le \norm{u_0}_{M^{2,1}}.
\]

Assume $u_0\in L^2_\uloc$.  
Since $u_0$ is $\la$-DSS then, for any $(x_0,R)$, we can take $k$ so that $\la^k\leq R <\la^{k+1}$, and we have
\[
\frac 1 {R}\int_{B_R(x_0)} |u_0(x)|^2\,dx\leq \frac 1 {\la^{k}}\int_{B_{\la^{k+1}}(x_0)} |u_0(x)|^2\,dx =  \la \int_{B_{1}(x_0/\la^{k+1})} |u_0(y)|^2\,dy \leq \la \|u_0\|_{L^2_{\uloc}}^2.
\]
Therefore, 
\[
\|u_0\|_{M^{2,1}} \le \sqrt \la \, \|u_0\|_{L^2_{\uloc}}.\qedhere
\]
\end{proof}

\begin{lemma}\label{weakLq-into-Morrey}
If $E \subset \R^n$, $|E|<\infty$, $1\le p < q<\infty$. Then $L^{q,\infty}(E) \subset L^p(E)$ and
\EQ{\label{eq0}
\int_E |f|^p \lec \norm{f}_{L^{q,\infty}(E)}^p |E|^{1-\frac pq}.
}
\end{lemma}
\begin{proof}
Let $\om(t) = \om_{f,E}(t) = |\{x \in E: |f(x)|>t\}|$, the distribution function of $|f|$ on $E$. Let $M=\norm{f}_{L^{q,\infty}(E)}$.
We have $\om(t)\le \min(|E|, \, (M/t)^q)$ and for $T>0$
\EQN{
\int_E |f|^p &= \int_0^\infty p t^{p-1} \om(t)\,dt
\\
&\le \int_0^T p t^{p-1} |E|\,dt+\int_T^\infty p t^{p-1} (M/t)^q\,dt
\\
&= T^p |E| + \frac p{q-p} M^q T^{p-q}
}
Choosing $T=M|E|^{-1/q}$, we get \eqref{eq0}.
\end{proof}

In particular, if $E=B_R$,
\[
\frac 1{R^m} \int_{B_R} |f|^p \lec M^p, \quad m = n(1-\frac pq).
\]
This shows
\EQ{\label{eq1}
L^{q,\infty} (\R^n) \subset M^{p, n(1-\frac pq)}(\R^n).
}

We limit ourselves to $\R^3$ in the following lemma.

\begin{lemma}\label{lemma62}
We have
\EQ{\label{eq2}
L^{3,\infty} (\R^3) \subset M^{2, 1}(\R^3)\subset L^2_\uloc,
}
\EQ{\label{eq3}
L^{3,\infty} (\R^3) \subset E^2,
}
but for any $\la>1$,
\EQ{\label{eq4}
\la\text{-DSS } \cap 
 M^{2, 1}(\R^3)\not\subset E^2.
}
\end{lemma}
\begin{proof}
The first inclusion of \eqref{eq2} follows from \eqref{eq1} with $n=3=q$ and $p=2$. The second inclusion follows from the definition of $M^{2,1}$ with $R=1$.

To show \eqref{eq3}, consider any $f \in L^{3,\infty}(\R^3)$. Let $M=\norm{f}_{L^{3,\infty}(\R^3)}$. 
For any $\e>0$, we can choose $R=R(\e)\gg 1$ such that
\[
  |\{x \in \R^3: |f(x)|>\e , \ |x|>R \}| < \e^6.
\]
For any $x_0$ with $|x_0|>R+1$, let
\[
E_< =B_1(x_0)\cap \{|f|\le \e \}, \quad E_>=B_1(x_0)\cap \{|f|> \e \}.
\]
By Lemma \ref{weakLq-into-Morrey}, we have
\EQN{
\int_{B_1(x_0)} |f|^2 &=\int_{E_<} |f|^2  + \int_{E_>} |f|^2
\\
&\le \e ^2 |B_1| + CM^2 |E_>|^{1/3}
\\
&\le C\e^2+CM^2\e^2.
}

To show \eqref{eq4}, consider
the example in \cite[(1.14)]{BT1}: 
\EQ{\label{eq6.6}
f(x) = \sum_{k \in \mathbb Z} \la^k f_0(\la^k x), \quad f_0(x)= \frac{1}{|x-x_0|} \chi(x-x_0),
}
where  $1+r<|x_0|<\la -r$ for some $r>0$, and $\chi$  is the characteristic function of the ball $B_r(0)$. It is in $M^{2,1}(\R^3)$ and is $\la$-DSS, but it is not in $E^2$.
\end{proof}

\emph{Remark.} The function $f$ given in \eqref{eq6.6} is not in $L^{3,\infty}(\R^3)$ as $L^{3,\infty}(\R^3)\subset E^2$.
Because the restrictions $f(\cdot+\la^{-k} x_0)|_{B_1(0)}$ are the same for all $k$ sufficiently large,  the oscillation of $f$ does not decay as considered in \cite{KwTs}.

\medskip

Our final fact is an example highlighting a subtle difference between $M^{2,1}$ and $L^{3,\I}$.  In Section \ref{sec.unique} we mentioned that
\[
\| e^{t\Delta}u_0\|_{L^4_\uloc}\lesssim t^{-1/8}\|u_0\|_{L^{3,\I}},
\]
but that this may fail if $L^{3,\I}$ is replaced by $M^{2,1}$. This complicated our adaptation of Jia's proof of uniqueness from \cite{Jia-uniqueness}. The following example confirms that this estimate does not hold generally when $u_0\in M^{2,1}$.
Note that, by \cite[Lemma 2.1]{Kato}, we have
\[
\norm{e^{t\Delta}u_0}_{L^4_\uloc} \lec \norm{e^{t\Delta}u_0}_{M^{4,1}} \lec 
t^{-1/4} \norm{u_0}_{M^{2,1}}.
\]
Thus $\norm{u(t)}_{L^4_\uloc}$ is always finite if $u_0 \in M^{2,1}$, and the estimate in question, \eqref{estimateQ}, may fail only for small $t$.

\medskip

\begin{example}\label{counterexample} 
Let $f$ be given by \eqref{eq6.6} and let $f_k(x) = \la^k f_0 (\la^k x)$.  Let $u_k = e^{t\Delta}f_k$ and $x_k = \la^{-k} x_0$. Notice that $u_k$ are all nonnegative.  
Assume that $t$ satisfies $\la^{-l} \leq \sqrt t < \la^{-l+1}$ where $l\in \N$.  
Also assume $1\leq k\leq l$ and $x\in B_l^{(k)} = B_{\la^{-l}} (x_k)$.
Consider 
\EQN{
\int \frac 1 {t^{3/2}} e^{-|y|^2/t} f_k(x-y) \,dy &=  \int \frac 1 {t^{3/2}} e^{-|y|^2/t} \frac 1 {|(x-y)-x_k|} \chi(\la^k(x-y)-x_0) \,dy.
}
Everything inside the integral is nonnegative and 
\[
 \chi(\la^k(x-y)-x_0)= \chi_{B_{r\la^{-k}} (x-x_k)}(y) \geq \chi_{B_{r\la^{-l}} (x-x_k)}(y),
\]
because $1\leq k\leq l$ and so
\[
B_{r/\la^{l}}(x-x_k)\subset B_{r/\la^k}(x-x_k).
\]
Hence
\EQN{
\int \frac 1 {t^{3/2}} e^{-|y|^2/t} f_k(x-y) \,dy  
&\geq   \int \frac 1 {t^{3/2}} e^{-|y|^2/t} \frac 1 {|(x-y)-x_k|}\chi_{B_{r\la^{-l}} (x-x_k)}(y) \,dy.
}
If $y\in B_{r/\la^{l}}(x-x_k)$, then
\[
|(x-y)-x_k| \leq C \la^{-l} \leq C \sqrt t.
\]
Hence, if $x\in B_l^{(k)}$, then 
\EQN{
\int \frac 1 {t^{3/2}} e^{-|y|^2/t} f_k(x-y) \,dy  
&\geq \frac C {\sqrt t}  \int_{B_{r\la^{-l}}(x-x_k) } \frac 1 {t^{3/2}} e^{-|y|^2/t}   \,dy 
\\&\geq \frac C {\sqrt t}  \int_{B_{r\la^{-l}/\sqrt t} (x/\sqrt t-x_k/\sqrt t) } e^{-|z|^2 }   \,dz 
\\&\geq \frac C {\la  \sqrt t} e^{-(r+1)^2} r^3 ,
}
where we let $z=y/\sqrt t$ and   used the fact that 
\EQN{ 
|z| \leq |z-(x-x_k)/\sqrt t| + |x-x_k|/\sqrt t\leq  \frac { r \la^{-l}} {\sqrt t} + \frac {\la^{-l}} {\sqrt t} \leq r+1,
} 
and 
\EQN{
|r\la^{-l}/\sqrt t|^3  \geq   \la^{-1} r^3.
}
Now, 
\EQ{
\int_{B_1} |e^{t\Delta} f|^4\,dx \gtrsim \sum_{k=1}^l \int_{B_l^{(k)}} |u_k(x,t)|^4 \,dx \gtrsim \sum_{k=1}^l \frac {| B_l^{(k)}|  } {\sqrt t ^4} \geq C(\la) \frac { |\log t|} {\sqrt t},
}
and it is therefore \emph{not} possible that 
\EQ{\label{estimateQ}
\| e^{t\Delta}f\|_{L^4_\uloc}\lesssim t^{-1/8}\|f\|_{M^{2,1}}.
}
As mentioned above, $f$ is obviously not in $L^{3,\I}$ because it is not in $E^2$. The above computations show it also fails to be in $L^{3,\I}$ locally. This illustrates how $M^{2,1}$ is locally weaker than $L^{3,\I}$.  This can also be checked directly. For $\si>0$ let
\[
E_\si = \bket{ |f(x)|>\si: \ x \in B_1}.
\]
For $l\in \N$, let $\si_l = \la^l r^{-1}$. Then 
\[
|E_{\si_l}| = \sum _{k=1}^l C  \la^{-3l} + \sum _{k>1}^l C \la^{-3k} \ge C l  \la^{-3l}.
\]
Thus
\[
\norm{f}_{L^{3,\infty}(B_1)} \ge 
\si_l |E_{\si_l}| ^{1/3} \ge C l^{1/3}.
\]
Taking $l \to \infty$, we get $\norm{f}_{L^{3,\infty}(B_1)} =\infty$. 
\end{example}

\section*{Acknowledgments}
The research of Tsai was partially supported by NSERC grant 261356-18 (Canada).

 Zachary Bradshaw, Department of Mathematics, University of Arkansas, Fayetteville, AR 72701, USA;
 e-mail: zb002@uark.edu
 \medskip
 
 Tai-Peng Tsai, Department of Mathematics, University of British
 Columbia, Vancouver, BC V6T 1Z2, Canada;
 e-mail: ttsai@math.ubc.ca

\end{document}